\documentclass[hyperref]{article}
\usepackage{float} 
\usepackage{amssymb}
\usepackage{amsmath}
\usepackage{mathtools}
\usepackage[thmmarks,amsmath]{ntheorem} 
\usepackage{bm} 
\usepackage[colorlinks,linkcolor=cyan,citecolor=cyan]{hyperref}
\usepackage{enumerate}
\usepackage{extarrows}
\usepackage[hang,flushmargin]{footmisc} 
\usepackage[square,comma,sort&compress,numbers]{natbib} 
\usepackage{mathrsfs} 
\usepackage[font=footnotesize,skip=0pt,textfont=rm,labelfont=rm]{caption,subcaption} 
\usepackage{booktabs} 
\usepackage{tocloft}
\usepackage{authblk}
\usepackage{hyperref}

\newtheorem{thm}{Theorem}[section]
\newtheorem{yl}[thm]{Lemma}

\newtheorem{mt}[thm]{Proposition}

\newtheorem{zj}{Remark}[section]
\newtheorem{dy}{Definition}[section]
{
	\theoremstyle{nonumberplain}
	\theoremheaderfont{\bfseries}
	\theorembodyfont{\normalfont}
	\theoremsymbol{\mbox{$\Box$}}
	\newtheorem{proof}{Proof}
}

\usepackage{theorem}

{
	\theoremheaderfont{\bfseries}
	\theorembodyfont{\normalfont}
	\newtheorem{remark}{Remark}[section]
}
\usepackage[a4paper]{geometry}                           
\begin{document}
	\title{\bf   Stability and energy identity for Yang-Mills-Higgs pairs}
	\date{}
    \author{\sffamily Xiaoli Han$^1$, Xishen Jin$^2$, Yang Wen$^{3^*}$\\
    	{\sffamily\small $^1$Math department of Tsinghua university, Beijing, 100084, China\\hanxiaoli@mail.tsinghua.edu.cn}\\
        {\sffamily\small $^2$School of Mathematics, Remin University of China, Beijing, 100872, China\\jinxishen@ruc.edu.cn}\\
    	{\sffamily\small $^{3^*}$Academy of Mathematics and Systems Science, The Chinese Academy of Science, Beijing, 100190, China \\wenyang19@mails.ucas.ac.cn}\\
       {\sffamily\small $^*$Corresponding author.}}

	\maketitle
	{\noindent\small{\bf Abstract:}
	\renewcommand{\thefootnote}{}
	\footnotetext{This is the revised version submitted on 2023.1.7}In this paper, we study the properties of the critical points of Yang-Mills-Higgs functional, which are called Yang-Mills-Higgs pairs. We first consider the properties of weakly stable Yang-Mills-Higgs pairs on a vector bundle over $S^n$ $(n\ge4)$. When $n \geq 4$, we prove that the norm of its Higgs field is $1$ and the connection is actually Yang-Mills. More precisely, its curvature vanishes when $n\geq 5$. We also use the bubble-neck decomposition to prove the energy identity of a sequence of Yang-Mills-Higgs pairs over a $4-$dimensional compact manifold with uniformly bounded energy. We show there is a subsequence  converges smoothly to a Yang-Mills-Higgs pair up to gauge modulo finitely many $4-$dimensional spheres with Yang-Mills connections.

 }
	
	\vspace{1ex}
	{\noindent\small{\bf Keywords:}
		Yang-Mills-Higgs pairs}
	\section{Introduction}\
	
  Let $\left( M,g \right) $ be an $n$-dimensional closed Riemannian  manifold and $E$ be a vector bundle of rank $r$ over $M$ with structure group $G$, where $G$ is a compact Lie group. Let $\mathfrak{g}_E$ be the adjoint bundle of $E$. The classical Yang-Mills functional defined on the space of connections of $E$ is given by
	\[YM\left( \nabla  \right)=\int_M |R^\nabla|^2dV \]
	where $\nabla$ is a connection on $E$, $R^\nabla$ denotes its curvature and $dV $ is the volume form of $g$. We denote $d^\nabla $ to be the exterior differential induced by $\nabla$ and $\delta^\triangledown$ be the formal adjoint of $d^\triangledown$. The critical points of Yang-Mills functional are called  Yang-Mills connections  and they satisfy
		\[\delta^\nabla R^\nabla =0. \]
	Our interests are  the Yang-Mills connections which minimize the Yang-Mills functional locally. At such a connection $\nabla$, the second variation of the Yang-Mills functional should be non-negative, i.e.
\[\left.\dfrac{d^2 }{dt^2 }YM\left( \nabla^t \right)\right|_{ t=0 }  \ge  0\]
where $\nabla^t $ is a curve of connections with $\nabla^0 =\nabla$. Such connections are called stable. Considering the second variation of the Yang-Mills functional with respect to deformations generated by special vector fields, J. Simons announced that every stable Yang-Mills connection on $S^n $ is flat, if $n > 4$ in Tokyo in September of 1977. Bourguignon-Lawson \cite{JL} gave a detailed proof of this result.

The Ginzburg-Landau equations are the Euler-Lagrange equation of the Ginzburg-Landau functional 
\[E_\varepsilon \left( u \right) =\int_M \left( \frac{\left|\nabla u\right|^2 }{2}+\frac{\left( 1-\left| u \right|^2 \right)^2 }{4\varepsilon^2 }  \right) dV \]
where $u$ is a complex-valued function on $M$. 
In \cite{Ch}, Cheng proved that every stable solutions of the Ginzburg-Landau equation on $S^n $ for $n \geq 2$ are the constant with absolute value $1$.

In this paper, we consider the following Yang-Mills-Higgs functional with self-interaction parameter $\lambda \geq 0$ as a combination of the Yang-Mills functional and Ginzburg-Landau functional 
\begin{equation}\label{YMH-End}
	\mathscr{A}(\nabla,u)=\frac12\int_M|R^\triangledown|^2+|d^\triangledown u|^2+\frac\lambda4(1-|u|^2)^2dV
\end{equation}
where $u \in \Omega^0 ( E) $ is a Higgs field. The Yang-Mills-Higgs functional on $ \mathbb{R}^3 $ with structure group $SU\left( 2 \right) $ was first introduced by P. Higgs in \cite{PHiggs}. A. Jaffe-C. Taubes \cite{JT} extended the Yang-Mills-Higgs functional to $\mathbb{R}^n $ and also general manifolds. Its critical point is the so-called magnetic monopole (we also call it a Yang-Mills-Higgs pair), i.e. a pair $\left( \nabla,u \right)$ satisfying
    \begin{equation}\label{YMH}
	\begin{split}
		\delta^\triangledown R^\triangledown&=-\frac12(d^\triangledown u\otimes u^\ast-u\otimes(d^\triangledown u)^\ast),\\
		\delta^\triangledown d^\triangledown u&=\frac\lambda2(1-|u|^2)u,
	\end{split}
\end{equation}
where for any $u\in\Omega^0(E)$ and $\phi\in\Omega^p(E)$, $\frac12(u\otimes\phi^\ast-\phi\otimes u^\ast)\in\Omega^p(\mathfrak{g}_E)$  such that for any $\varphi\in\Omega^p(\mathfrak{g}_E)$, we have $\langle \frac12(u\otimes\phi^\ast-\phi\otimes u^\ast),\varphi\rangle =-\langle \phi,\varphi u\rangle $.

Similar to the stable Yang-Mills connection, a Yang-Mills-Higgs pair $\left( \nabla,u \right)$ is called stable if for any curve $\left( \nabla^t,u^t \right) $ such that $\nabla^0 =\nabla $ and $u^0 =u $, there holds
\begin{equation}
	\left.\frac{d^2}{dt^2}\mathscr{A}(\nabla^t,u^t)\right|_{t=0}\ \geq 0.
\end{equation} 
Furthermore, we can define the notation of the weakly stable of Yang-Mills-Higgs pairs (c.f. Definition \ref{weakly stable}). The purpose of the present work is to extend some of the results of J. Simons and Bourguignon-Lawson \cite{JL} about weakly stable Yang-Mills connections on $S^n $ to the weakly stable Yang-Mills-Higgs pairs. We prove the following theorem.

	\begin{thm}\label{mainthm1}
	Assume $(\nabla,u)$ is a weakly stable Yang-Mills-Higgs pairs on $S^n$ , then
	\begin{enumerate}
		\item If $n\ge 5$, then $R^\triangledown=0$, $d^\triangledown u=0$ and  $|u|=1$.
		\item 	If $n=4$, then $d^\triangledown u=0$, $|u|=1$ and $\nabla$ is a Yang-Mills connection (i.e. $\delta^\triangledown R^\triangledown=0$).
	\end{enumerate}
    \end{thm}
    
    The Higgs fields taking values in $\Omega^0(\mathfrak{g}_E)$ are sometimes concerned in some physical research. The corresponding Yang-Mills-Higgs functional is
    \begin{align*}
    	\mathscr{A}(\nabla,\Phi)=\frac12\int_M|R^\triangledown|^2+|d^\triangledown\Phi|^2+\frac\lambda4(1-|\Phi|^2)^2dV,
    \end{align*}
    where $\Phi\in\Omega^0(\mathfrak{g}_E)$. The Euler-Lagrange equation of $\mathscr{A}$ is
    \begin{align}\label{YMH2}
	    \begin{split}
		    \delta^\triangledown R^\triangledown&=[d^\triangledown\Phi,\Phi],\\
		    \delta^\triangledown d^\triangledown \Phi&=\frac\lambda2(1-|\Phi|^2)\Phi.
	    \end{split}
    \end{align}
   The stable Yang-Mills-Higgs pairs $(\nabla,\Phi)$ on $S^n$ have similar properties, which are discussed in Section 4.
   
    The energy identity was first established in \cite{DK} in dimension $4$ manifolds for sequences of anti-self-dual Yang-Mills fields. In \cite{Tian}, Tian proved that the defect measure of sequences of Yang-Mills fields on a Riemannian manifold $(M, g)$ of dimension $n$ ($n\geq 4$) is carried by a $n-4$-rectifiable subset $S$ of $M$. Rivi\'{e}re  \cite{Riviere} proved that, in $4$ dimension, at any point of $S$ it is the sum of $L^2$ energies of Yang-Mills fields on $S^4$ and this result holds on any dimension under the additional assumption on the $W^{2,1}$ norm of the curvature.  Moreover, in \cite{NV}, Nabor-Valtorta proved that the $W^{ 2,1 }$-norm is automatically bounded for a sequence of stationary Yang-Mills fields with bounded energy.
    
     In \cite{Song}, Song proved the energy identity for a sequence of Yang-Mills-Higgs pairs on a fiber bundle with curved fiber
     spaces over a compact Riemannian surface and the blow-up only occurs in the Higgs part in the 2-dimensional case. In Section 5, we assume that $M$ is a 4-dimensional compact Riemannian manifold and $\{(\nabla_i,u_i)\}$ is a sequence of Yang-Mills-Higgs pairs over $M$ with uniformly bounded energy $\mathscr{A}(\nabla_i,u_i)\le K$. Unlike the case of 2-dimensional manifolds, it will be shown that there is no energy concentration point for Higgs field over $4-$dim manifolds and the blow-up only occurs in the curvature part. We prove the following theorem.
   \begin{thm}\label{maintheorem2}
   	Assume $\{(\nabla_i,u_i)\}$ is a family of Yang-Mills-Higgs pairs which satisfy the equation (\ref{YMH2}) and $\mathscr{A}(\nabla_i,u_i)\le K$. Then there is a finite subset $\Sigma=\{x_1,...,x_l\}\subset M$, a Yang-Mills-Higgs pair $(\nabla_\infty,u_\infty)$ on $M\setminus\Sigma$ and Yang-Mills connections  $\{\widetilde\triangledown_{jk}\mid1\le j\le l,1\le k\le K_j\}$ over $S^4$, such that there is a subsequence of $\{(\nabla_i,u_i)\}$ converges to $(\nabla_\infty,u_\infty)$ in $C^\infty_{loc}(M\setminus\Sigma)$ under gauge transformations and
   	\begin{equation}
   		\lim_{i\to\infty}\mathscr{A}(\nabla_i,u_i)=\mathscr{A}(\nabla_\infty,u_\infty)+\sum_{j=1}^l\sum_{k=1}^{K_j}YM(\widetilde\nabla_{jk}).
   	\end{equation}
   \end{thm}
   
	\section{Preliminary}\ 
	
	Let $\left( M,g \right) $ be a compact Riemannian manifold and $D$ be its Levi-Civita connection. Let $E\to M$ be a rank $r$ vector bundle over $M $ with a compact Lie group $G\subset SO\left( r \right) $ as its structure group. We also assume $\langle \ ,\ \rangle$ is a Riemannian metric of $E$ compatible with the action of $G$ . Let $\mathfrak{g}_E $ be the adjoint bundle of $E$. Assume $\nabla:\Omega^0(E)\to\Omega^1(E)$ is a connection on $E$ compatible with the metric $\langle \ ,\ \rangle$. Locally, $\nabla$ takes the form
	\[\nabla=d+A\]
	where $A\in \Omega^1 \left( \mathfrak{g}_E \right) $ is the connection $1$-form.
	
	For any connection $\nabla$ of $E$, the curvature $R^\nabla=\nabla^2 \in \Omega^2 \left( \mathfrak{g}_E \right) $  measures the extent to which $\nabla$ fails to commute. Locally, $R^\nabla $ is given by
	\[R^\nabla =dA +\frac{1}{2}[A\wedge A]\]
	where the bracket of $\mathfrak{g}_E $-valued $1$-forms $\varphi$ and $\psi$ is defined to be 
	\[[\varphi \wedge \psi ]_{ X,Y }=\left[ \varphi_X,\psi_Y \right] -\left[ \varphi_Y, \psi_X \right]\]
	as in \cite{JL}. 
	
	The connection $\nabla$ on $E$ induces a natural connection on $\mathfrak{g}_E$. Indeed for $\phi\in \Omega^0(\mathfrak{g}_E)$ we define
	\[\nabla(\phi)=[\nabla,\phi],\] i.e. $\nabla(\phi)(\sigma)=\nabla(\phi(\sigma))-\phi(\nabla\sigma)$ for any section $\sigma$ of $E$. By direct calculation, for any $X\in T(M)$ we have
	\begin{align*}
		\nabla_X(\frac12(u\otimes\phi^\ast-\phi\otimes u^\ast))=\frac12(\nabla_Xu\otimes\phi^\ast-\phi\otimes (\nabla_Xu)^\ast)+\frac12(u\otimes(\nabla_X\phi)^\ast-\nabla_X\phi\otimes u^\ast).
	\end{align*}
Similarly, the curvature of this connection on $\mathfrak{g}_E$ is given by the formula
	\[R^\nabla_{X, Y}(\phi)=[R^\nabla_{X, Y},\phi],\] where $R^\nabla$ on the right denotes the curvature of $E$. We define an exterior differential $d^\nabla:\Omega^p(\mathfrak{g}_E)\rightarrow\Omega^{p+1}(\mathfrak{g}_E),p\geq 0$ as follows. For each real-valued differential $p$-form $\alpha$ and each section $\sigma$ of $E$, we set
	\[d^\nabla(\alpha\otimes\sigma)=(d\alpha)\wedge\sigma+(-1)^p\alpha\otimes\nabla\sigma\] and extend the definition to general $\psi\in \Omega^p(E)$ by linearity. Note that $d^\nabla=\nabla$ on $\Omega^0(E)$ and $(d^\nabla (d^\nabla\sigma))_{X, Y}=R^\nabla_{X, Y}(\sigma)$.

     The inner product of $\Omega^0 \left( \mathfrak{g}_E \right) $ induced by the trace inner product metric on $\mathfrak{so}\left( r \right) $ is given by
    \[\langle  \phi,\varphi\rangle =\frac{1}{2}\operatorname{Tr}\left( \phi^T \varphi \right) ,\text{ where } \varphi,\phi \in \Omega^0 \left( \mathfrak{g}_E \right) . \] 
    Then for any $\phi,\varphi,\rho\in\Omega^0(\mathfrak{g}_E)$, we have \[\langle [\phi,\varphi],\rho\rangle=\langle \phi,[\varphi,\rho]\rangle.\] Combining with Riemannian metric $g$, the inner product of $\Omega^p(\mathfrak{g}_E)$ can be defined by
    \begin{equation}
    	\langle \phi,\varphi\rangle=\frac{1}{p!}\sum_{1\le i_1,...,i_p\le n}\langle \phi(e_{i_1},...e_{i_p}),\varphi(e_{i_1},...,e_{i_p})\rangle
    \end{equation}
    where $\phi,\varphi\in\Omega^p(\mathfrak{g}_E) $ and $\{e_i\mid1\le i\le n\}$ is an orthogonal basis of $TM$.
	Integrating the inner product $\langle \ ,\ \rangle$ over $M$, we get a global inner product $\left( \ ,\  \right)$  in $\Omega^p \left( \mathfrak{g}_E \right)$, i.e. 
	\[\left( \varphi,\psi \right)=\int_M \langle  \varphi,\psi \rangle dV, \text{ for any }\varphi,\psi\in \Omega^p \left( \mathfrak{g}_E \right) .\]
	Define the operator $\delta^\nabla:\Omega^{p+1}(\mathfrak{g}_E)\rightarrow \Omega^p(\mathfrak{g}_E), p\geq 0,$ to be the formal adjoint of the operator $d^\nabla$. In local coordinates, for any $\phi\in \Omega^p(\mathfrak{g}_E)$,
	\[(d^\nabla\phi)_{X_0,\cdots,X_p}=\sum_{k=0}^p(-1)^k(\nabla_{X_k}\phi)_{X_0,\cdots,\widehat{X_k},\cdots, X_p},\]
	\[(\delta^\nabla\phi)_{X_1,\cdots, X_{p-1}}=-\sum_{j=1}^n(\nabla_{e_j}\phi)_{e_j, X_1,\cdots,X_{p-1}},\] where $\{e_i\}$ is an orthonormal basis of $TM$.

    We can define the Laplace-Beltrami operator $\Delta^\nabla $ by
	\[\Delta^\triangledown=d^\triangledown\delta^\triangledown+\delta^\triangledown d^\triangledown\] 
	and the rough Laplacian operator $ \nabla^* \nabla $ by \[\nabla^\ast\nabla=-\sum_j(\nabla_{e_j}\nabla_{e_j}-\nabla_{D_{e_j}e_j}).\] 
	For $\psi\in\Omega^1(\mathfrak{g}_E)$ and $\varphi\in\Omega^2(\mathfrak{g}_E)$, we recall the following operator $\mathfrak{R}^\triangledown$ defined in \cite{JL}
	\begin{align}
		&\mathfrak{R}^\triangledown(\psi)(X)=\sum_j[R^\triangledown(e_j,X),\psi(e_j)],\\
		&\mathfrak{R}^\triangledown(\varphi)(X,Y)=\sum_j[R^\triangledown(e_j,X),\varphi(e_j,Y)]-[R^\triangledown(e_j,Y),\varphi(e_j,X)].
	\end{align}
	
	Then we have the following Bochner-Weizenb{\" o}ck formula first introduced in \cite{JL}.
	
	\begin{thm}
		For any $\psi\in\Omega^1(\mathfrak{g}_E)$ and $\varphi\in\Omega^2(\mathfrak{g}_E )$, we have
		\begin{align}\label{Bochner}
			&\Delta^\triangledown\psi=\nabla^\ast\nabla\psi+\psi\circ \operatorname{Ric}+\mathfrak{R}^\triangledown(\psi),\\
			&\Delta^\triangledown \varphi=\nabla^\ast\nabla\varphi+\varphi\circ(\operatorname{Ric}\wedge \operatorname{Id}+2R_M)+\mathfrak{R}^\triangledown(\varphi),
		\end{align}
		where 
		\begin{itemize}
			\item $\operatorname{Ric} :TM\to TM$ is the Ricci transformation defined by
			\[\operatorname{Ric}\left( X \right) =\sum_jR(X,e_j)e_j ,\]
			\item $\psi\circ \operatorname{Ric }\in \Omega^1 \left( \mathfrak{g}_E \right) $ and $\left( \psi\circ \operatorname{Ric } \right) _X =\psi_{ \operatorname{Ric} \left( X \right) } \in \Omega^0 \left( \mathfrak{g}_E \right) ,$
			\item $ \operatorname{Ric}\wedge \operatorname{Id} $ is the extension of the Ricci transformation $\operatorname{Ric}$ to $\wedge^2 TM$ given by
			\[ \left( \operatorname{Ric}\wedge \operatorname{Id} \right)_{ X,Y }=\left( \operatorname{Ric}\wedge \operatorname{Id} \right) \left( X\wedge Y \right)  =\operatorname{Ric}(X)\wedge Y+X\wedge \operatorname{Ric}(Y),\]
			\item $R_M$ is the curvature of $TM$,
			\item For any map $\omega:\wedge^2 TM\to \wedge^2 TM$, the composite map $\varphi\circ \omega :\wedge^2 TM \to \Omega^0 \left( \mathfrak{g}_E \right)  $ is defined by
			\[\left( \varphi\circ \omega \right)_{ X,Y }=\left( \varphi\circ \omega \right)\left( X\wedge Y \right)  =\frac{1}{2}\sum_{j=1}^{n} \varphi_{ e_j, \omega_{ X,Y }e_j } .\]
		\end{itemize}
	\end{thm}
	
	\begin{zj}
   In particular, on the standard sphere $S^n $,  \[\operatorname{Ric}\left( X \right) =\left( n-1 \right) X\]
   and 
   \[(R_M)_{ X,Y }Z=-\left( X\wedge Y \right)  \left( Z \right)  =\langle  Y,Z\rangle X-\langle  X,Z\rangle Y.\]
    Thus, we have
   \[\psi\circ \operatorname{Ric}=\left( n-1 \right) \psi \]
   and 
   \[\varphi\circ \left( \operatorname{Ric} \wedge \operatorname{Id}+2R_M \right) =2\left( n-2 \right) \varphi.\]
\end{zj}
\

Note that for any $B\in \Omega^1 \left( \mathfrak{g}_E\right) $ and $w \in \Omega^0 \left( E	 \right) $,
\begin{equation}\label{curvature-t}
	R^{\triangledown+tB}=R^\nabla +td^\triangledown B +t^2 B\wedge B
\end{equation}
and
\begin{equation}\label{dPhi-t}
	d^{\triangledown+tB}\left( u+tw \right)=d^\triangledown u+t(B\cdot u+d^\triangledown w)+t^2B\cdot w .
\end{equation}

Assume $(\nabla,u)$ is a Yang-Mills-Higgs pair satisfying the equation (\ref{YMH}) and $(\nabla^t,u^t)$ is a curve on $$\tilde{\mathcal{A}}=\{(\widetilde\nabla,\tilde u)\mid \widetilde\nabla \textrm{ is a connection of }E,\ \tilde u\in\Omega^0(E)\}$$ such that $\nabla^0=\nabla$ and $u^0=u$. If we assume $\frac{d}{dt}(\nabla^t,u^t)\mid_{t=0}=(B,w)$, the second variation of $\mathscr{A}$ is
\begin{align}
	\begin{split}
		\frac{d^2}{dt^2}\mathscr{A}(\nabla^t,u^t)\mid_{t=0}=\int_M&\langle  \delta^\triangledown d^\triangledown B+\mathfrak{R}^\triangledown(B)+\frac12 (Bu\otimes u^\ast-u\otimes(Bu)^\ast),B\rangle+2\langle   d^\triangledown u,Bw\rangle\\
		&+2\langle     d^\triangledown w,Bu\rangle+\langle \delta^\triangledown d^\triangledown w+\lambda\langle  u,w\rangle u-\frac\lambda2(1-|u|^2)w,w\rangle dV.
	\end{split}
\end{align}
For any gauge transformation $g\in\mathcal{G}$, $g$ acts on $(\nabla,u)$ such that $(\nabla^g,u^g)=(g\circ\nabla\circ g^{-1},gu)$. Then for any $\sigma\in\Omega^0(\mathfrak{g}_E)$, let $g_t=\exp(t\sigma)$ be a family of gauge transformations, the variation of $(\nabla,u)$ along $\sigma$ is
\begin{align*}
	\frac{d}{dt}(\nabla^{g_t},u^{g_t})\mid_{t=0}=(-d^\triangledown\sigma,\sigma u)\in T_{ \left( \nabla,u \right)  }\tilde{\mathcal{A}}=\Omega^1 \left( \mathfrak{g}_E \right) \times \Omega^0 \left(E \right).
\end{align*}
In addition, the Yang-Mills-Higgs functional $\mathscr{A}$ is invariant under the action of gauge group. So it is interesting to consider the variation $\left( B,w \right)$  perpendicular to the direction of gauge transformation with respect to the global inner product on $\Omega^1 \left( \mathfrak{g}_E \right) \times \Omega^0 \left( \mathfrak{g}_E \right)$. Define
\begin{equation*}
	\begin{split}
 \zeta:\Omega^0 (\mathfrak{g}_E) &\to \Omega^1 (\mathfrak{g}_E)\times\Omega^0(\mathfrak{g}_E) \\
 \sigma  &\mapsto(-d^\triangledown\sigma,\sigma u),
	\end{split}
\end{equation*}
then $(B,w)\in \operatorname{Im}(\zeta)^\perp$ if and only if for any $\sigma$, 
\begin{align*}
	0=\int_M\langle  -d^\triangledown\sigma,B \rangle+\langle  \sigma u,w \rangle dV=\int_M\langle  -\delta^\triangledown B+\frac12(w\otimes u^\ast-u\otimes w^\ast),\sigma \rangle  dV.
\end{align*}
Thus the space of admissible variation at $\left( \nabla ,u \right) $ is 
\begin{equation}
	\mathcal{C}=T_{ \left( \nabla ,u \right)  } \left( \tilde{\mathcal{A}}/\mathcal{G} \right) =\operatorname{Im}(\zeta)^\perp=\{(B,w)\in\Omega^1(\mathfrak{g}_E)\times\Omega^0(E)\mid \delta^\triangledown B=\frac12(w\otimes u^\ast-u\otimes w^\ast)\}.
\end{equation}
Using $\delta^\triangledown(Bu)=(\delta^\triangledown B)u-B\llcorner d^\triangledown u$, where
\begin{align*}
	B\llcorner d^\triangledown u=\sum_jB(e_j)\nabla_{e_j}u,
\end{align*}
 and for
$(B,w)\in\mathcal{C}$, we have
\begin{align*}
	\int_M2\langle  d^\triangledown w,Bu \rangle  dV=\int_M2\langle  w,\delta^\triangledown(Bu) \rangle  dV=\int_M\langle  (w\otimes u^\ast-u\otimes w^\ast),\delta^\triangledown B \rangle  -2\langle  w,B\llcorner d^\triangledown u \rangle  dV.
\end{align*}
For any $x\in M$, let $\{e_i\mid 1\le i\le n\}$ be an orthogonal basis of $T_xM$. Since $B(e_i)\in\mathfrak{so}(E_x)$, we have 
$$-2\langle  w,B\llcorner d^\triangledown u \rangle  =\sum_i-2\langle  w,B(e_i)\nabla_{e_i}u \rangle  =\sum_i2\langle  B(e_i)w,\nabla_{e_i}u \rangle  =2\langle  Bw,d^\triangledown u \rangle  .$$
Hence for $(B,w)\in\mathcal{C}$, the second variation of $\mathscr{A}$ is
\begin{align}
	\begin{split}
		\frac{d^2}{dt^2}\mathscr{A}(\nabla^t,u^t)\mid_{t=0}=\int_M&\langle  \Delta^\triangledown B+\mathfrak{R}^\triangledown(B)+\frac12(Bu\otimes u^\ast-u\otimes(Bu)^\ast)+d^\triangledown u\otimes w^\ast-w\otimes(d^\triangledown u)^\ast,B \rangle  \\
		&+\langle  \delta^\triangledown d^\triangledown w+\frac12(w\otimes u^\ast-u\otimes w^\ast)u-2B\llcorner d^\triangledown u+\lambda\langle  u,w \rangle  u-\frac\lambda2(1-|u|^2)w,w \rangle  dV.
	\end{split}
\end{align}
Define an operator
\begin{align}
	\begin{split}
		\mathscr{S}^{(\nabla,u)}:\Omega^1(\mathfrak{g}_E)\times\Omega^0(E)&\to\Omega^1(\mathfrak{g}_E)\times\Omega^0(E),	
	\end{split}
\end{align}
where
\begin{align*}
	\begin{split}
		\mathscr{S}^{(\nabla,u)}(B,w)=&(\Delta^\triangledown B+\mathfrak{R}^\triangledown(B)+\frac12(Bu\otimes u^\ast-u\otimes(Bu)^\ast)+d^\triangledown u\otimes w^\ast-w\otimes(d^\triangledown u)^\ast,\\
		&\delta^\triangledown d^\triangledown w+\frac12(w\otimes u^\ast-u\otimes w^\ast)u-2B\llcorner d^\triangledown u+\lambda\langle  u,w \rangle  u-\frac\lambda2(1-|u|^2)w).
	\end{split}
\end{align*}
It is easy to see that $\mathscr{S}^{(\nabla,u)}$ is a self-adjoint operator on $\Omega^1(\mathfrak{g}_E)\times\Omega^0(E)$. Furthermore, we can prove that $\mathscr{S}^{(\nabla,u)}$ is a self-adjoint operator on $\mathcal{C}$.
\begin{yl}
	$\mathscr{S}^{(\nabla,u)}(\mathcal{C})\subset\mathcal{C}$.
\end{yl}
\begin{proof}
	Denote $\mathscr{S}^{(\nabla,u)}=(\mathscr{S}_1,\mathscr{S}_2)$. We only need to prove that for any $(B,w)\in\mathcal{C}$ and $\sigma\in\Omega^0(\mathfrak{g}_E)$, we have
	\begin{align*}
		\int_M\langle \mathscr{S}_1(d^\triangledown\sigma,-\sigma u),B\rangle +\langle \mathscr{S}_2(d^\triangledown\sigma,-\sigma u),w\rangle dV=0.
	\end{align*}
	First we have 
	$$\Delta^\triangledown d^\triangledown\sigma=d^\triangledown\delta^\triangledown d^\triangledown\sigma+\delta^\triangledown d^\triangledown d^\triangledown\sigma=d^\triangledown\Delta^\triangledown\sigma+\delta^\triangledown[R^\triangledown,\sigma]=d^\triangledown\Delta^\triangledown\sigma+[\delta^\triangledown R^\triangledown,\sigma]-\mathfrak{R}^\triangledown(d^\triangledown\sigma).$$
	The equation (\ref{YMH}) implies that for any $\phi\in\Omega^1(\mathfrak{g}_E)$, we have 
	$$\langle d^\triangledown u,\phi u\rangle =\langle \frac12(d^\triangledown u\otimes u^\ast-u\otimes (d^\triangledown u)^\ast),\phi\rangle =-\langle \delta^\triangledown R^\triangledown,\phi\rangle .$$ 
	By direct calculation, we have
	\begin{align*}
		&\int_M\langle \frac12((d^\triangledown\sigma\cdot u)\otimes u^\ast-u\otimes(d^\triangledown\sigma\cdot u)^\ast)+\sigma u\otimes(d^\triangledown u)^\ast-d^\triangledown u\otimes(\sigma u)^\ast,B\rangle dV\\
		=&\int_M\langle d^\triangledown\sigma\cdot u,Bu\rangle -2\langle d^\triangledown u,B\sigma u\rangle dV\\
		=&\int_M\langle \sigma u,\delta^\triangledown(Bu)\rangle -\langle \sigma\cdot d^\triangledown u,Bu\rangle -2\langle d^\triangledown u,B\sigma u\rangle dV\\
		=&\int_M\langle \sigma u,\delta^\triangledown B\cdot u\rangle -\langle \sigma u,B\llcorner d^\triangledown u\rangle -\langle \sigma\cdot d^\triangledown u,Bu\rangle -2\langle d^\triangledown u,B\sigma u\rangle dV\\
		=&\int_M\langle \sigma u,\delta^\triangledown B\cdot u\rangle +\langle d^\triangledown u,[\sigma,B]u\rangle dV\\
		=&\int_M\langle \sigma u,\delta^\triangledown B\cdot u\rangle -\langle \delta^\triangledown R^\triangledown,[\sigma, B]\rangle dV.
	\end{align*}
	Thus 
	$$\int_M\langle \mathscr{S}_1(d^\triangledown\sigma,-\sigma u),B\rangle dV=\int_M\langle \delta^\triangledown d^\triangledown\sigma,\delta^\triangledown B\rangle +\langle \sigma u,\delta^\triangledown B\cdot u\rangle dV.$$
	On the other hand, $\sigma\in\Omega^0(\mathfrak{g}_E)$ implies $\langle u,\sigma u\rangle =0$. By the equation (\ref{YMH}) we have
	\begin{align*}
		-\delta^\triangledown d^\triangledown(\sigma u)=-\delta^\triangledown d^\triangledown\sigma\cdot u+2d^\triangledown\sigma\llcorner d^\triangledown u-\frac\lambda2(1-|u|^2)\sigma u.
	\end{align*}
	Hence by $(B,w)\in\mathcal{C}$, we have
	\begin{align*}
		&\int_M\langle \mathscr{S}_2(d^\triangledown\sigma,-\sigma u),w\rangle dV\\
		=&\int_M-\langle \delta^\triangledown d^\triangledown\sigma\cdot u,w\rangle -\langle \sigma u,\frac12(w\otimes u^\ast-u\otimes w^\ast)u\rangle dV\\
		=&\int_M-\langle \delta^\triangledown d^\triangledown\sigma,\delta^\triangledown B\rangle -\langle \sigma u,Bu\rangle dV\\
	\end{align*}
	Then we have $\int_M\langle \mathscr{S}^{(\nabla,u)}(d^\triangledown\sigma,-\sigma u),(B,w)\rangle dV=0$ and we finish the proof.
\end{proof}

Then $\mathscr{S}^{(\nabla,u)}\mid_\mathcal{C}: \mathcal{C} \to\mathcal{C}$ is a self adjoint and elliptic operator. The eigenvalues of $\mathscr{S}^{(\nabla,u)}$ are given by
 \[\lambda_1  \leq \lambda_2 \leq \cdots \to +\infty. \]
 Similar as in \cite{JL}, we define the weakly stability of Yang-Mills-Higgs functional at $\left( \nabla,u \right) $ as following.

\begin{dy}\label{weakly stable}
	Assume $(\nabla,u)$ satisfies (\ref{YMH}), then it is called weakly stable if $\lambda_1\ge0$ and stable if $\lambda_1>0$.
\end{dy} 

\section{Stability of Yang-Mills-Higgs pairs on $S^n$} 

In this section, we prove Theorem \ref{mainthm1}. Now, we assume $\left( M,g \right) $ is the standard Euclidean sphere $S^n $. In \cite{LS}, the conformal Killing vector fields of $S^n $ play an important role in studying the non-existence of stable varifolds or currents. Similar methods have been applied to study weakly stable Yang-Mills connections on $S^n $ in \cite{JL}. The conformal Killing vector fields of $S^n $ are the gradients of eigenfunctions corresponding to the first non-zero eigenvalue of the Laplace operator. Let us summarize the properties of these vector fields as follow.

\begin{mt}\label{Killing}
	For any $v=(v^1,...,v^{n+1})\in\mathbb{R}^{n+1}$, let $F_v(x):\mathbb{R}^{n+1}\to\mathbb{R}$, $x\mapsto v\cdot x$ be the inner product of $v$ and $x$. Define $f_v=F_v\mid_{S^n}$ to be the restriction of $F_v$ to $S^n$. Then \begin{equation}\label{Vector}
		V=\operatorname{grad}(f_v)=\sum_{i=1}^{n+1}(v^j-(v\cdot x)x^j)\frac{\partial}{\partial x^j}=v-\left( v\cdot x \right)x 
	\end{equation} 
is a conformal Killing vector fields on $S^n$ and satisfies
	\begin{enumerate}[(1).]
		\item $ D_X V=-f_v X,  $
		\item $ D^\ast DV=V. $
	\end{enumerate}
\end{mt}

\begin{zj}
	In fact, the space $\mathscr{V}$ of all $V$ defined above is the orthogonal complement to the Killing vector fields in the space of all conformal vector fields on $S^n $, i.e.
	\[\mathfrak{conf}\left( S^n  \right) =\mathfrak{isom}\left( S^n \right)\oplus \mathscr{V} .\]
\end{zj}

Similar as in \cite{JL}, we choose the variation of connection to be $B_v=i_VR^\triangledown$, where $i_V$ is the contraction about $V$. The corresponding variation of the Higgs field $w_v$ satisfies
\begin{align*}
	\frac12(w_v\otimes u^\ast-u\otimes w_v^\ast)=\delta^\triangledown i_VR^\triangledown=-\delta^\triangledown R^\triangledown(V)=\frac12(\nabla_V u\otimes u^\ast-u\otimes(\nabla_Vu)^\ast).
\end{align*}
Hence, $w_v=\nabla_Vu$ satisfies $(B_v,w_v)\in\mathcal{C}$. If we define a quadratic form $\mathscr{L}^{(\nabla,u)}$ on $\mathcal{C}$ by setting
\begin{equation}\label{L}
	\mathscr{L}^{(\nabla,u)}(B,w)=\int_M\langle  \mathscr{S}^{(\nabla,u)}(B,w),(B,w) \rangle  dV, 
\end{equation}
then $\left.\frac{d^2}{dt^2}\mathscr{A}(\nabla^t,u^t)\right|_{t=0}=\mathscr{L}^{(\nabla,u)}(B,w)$ for any $(B,w)\in\mathcal{C}$.  $Q(v_1,v_2)=\mathscr{L}^{(\nabla,u)}(B_{v_1},w_{v_2})$ can be viewed as a quadratic form.
\begin{yl}\label{L-VV}
	Assume $(\nabla,u)$ is a Yang-Mills-Higgs pair on $S^n$. Then for any $v\in\mathbb{R}^{n+1}$, we have
	$$\mathscr{L}^{(\nabla,u)}(i_VR^\triangledown,\nabla_Vu)=\int_{S^n}(4-n)|i_VR^\triangledown|^2+(2-n)|\nabla_Vu|^2-2f_v(\langle i_VR^\triangledown,\delta^\triangledown R^\triangledown\rangle +\langle \delta^\triangledown d^\triangledown u,\nabla_Vu\rangle )dV.$$
\end{yl}
\begin{proof}
	Denote $\mathscr{S}^{(\nabla,u)}=(\mathscr{S}_1,\mathscr{S}_2)$. For any $x\in S^n$, let $\{e_i\mid 1\le i\le n\}$ be a local orthonormal frame of $S^n $ near $x$. According to (\ref{Bochner}), we have
	\begin{align*}
			\Delta^\triangledown i_VR^\triangledown(e_k)=\nabla^\ast\nabla i_VR^\triangledown(e_k)+(n-1)R^\triangledown(V,e_k)+\mathfrak{R}^\triangledown(i_VR^\triangledown)(e_k).
		\end{align*}
	Since $De_k(x)=0$, at $x$ 
	\begin{align*}
			\nabla^\ast\nabla i_VR^\triangledown(e_k)&=-\sum_j\nabla_{e_j}\nabla_{e_j}i_VR^\triangledown(e_k)\\
			&=-\sum_j\nabla_{e_j}(\nabla_{e_j}i_VR^\triangledown(e_k))\\
			&=-\sum_j\nabla_{e_j}(\nabla_{e_j}(R^\triangledown(V,e_k))-R^\triangledown(V,D_{e_j}e_k))\\
			&=-\sum_j\nabla_{e_j}(\nabla_{e_j}R^\triangledown(V,e_k)-f_vR^\triangledown(e_j,e_k))\\
			&=\nabla^\ast\nabla R^\triangledown(V,e_k)-2f_v\delta^\triangledown R^\triangledown(e_k)+R^\triangledown(V,e_k)\\
			&=\Delta^\triangledown R^\triangledown(V,e_k)+(5-2n)R^\triangledown(V,e_k)-\mathfrak{R}^\triangledown(R^\triangledown)(V,e_k)-2f_v\delta^\triangledown R^\triangledown(e_k),
		\end{align*}
	where we have used (1) of Proposition \ref{Killing}. By $2\mathfrak{R}^\triangledown(i_VR^\triangledown)(e_k)-\mathfrak{R}^\triangledown(R^\triangledown)(V,e_k)=0$, we have 
	$$\Delta^\triangledown i_VR^\triangledown(e_k)+\mathfrak{R}^\triangledown(i_VR^\triangledown)(e_k)=\Delta^\triangledown R^\triangledown(V,e_k)+(4-n)R^\triangledown(V,e_k)-2f_v\delta^\triangledown R^\triangledown(e_k).$$ 
	Similarly, by the definition of curvature and $[X,Y]=D_XY-D_YX$,  we have
	\begin{align*}
			\delta^\triangledown d^\triangledown\nabla_Vu=&-\sum_j\nabla_{e_j}\nabla_{e_j}\nabla_Vu\\
			=&-\sum_j\nabla_{e_j}(R^\triangledown(e_j,V)u+\nabla_V\nabla_{e_j}u+\nabla_{[e_j,V]}u)\\
			=&\delta^\triangledown R^\triangledown(V)u-2\mathfrak{R}^\triangledown(d^\triangledown u)(V)\\
			&-\sum_j\nabla_V\nabla_{e_j}\nabla_{e_j}u-2\nabla_{[e_j,V]}\nabla_{e_j}u-[R^\triangledown(e_j,[e_j,V]),u]-\nabla_{[e_j,[e_j,V]]}u.
	\end{align*}
	At $x$, we have 
	$$R^\triangledown(e_j,[e_j,V])=0,$$ 
	\begin{align*}
			-\sum_j\nabla_V\nabla_{e_j}\nabla_{e_j}u&=\nabla_V\delta^\triangledown d^\triangledown u-\sum_j\nabla_V\nabla_{D_{e_j}e_j}u\\
			&=\nabla_V\delta^\triangledown d^\triangledown u-\sum_j \left( R^\triangledown(V,D_{e_j}e_j)u+\nabla_{D_{e_j}e_j}\nabla_Vu+\nabla_{[V,D_{e_j}e_j]}u \right) \\
			&=\nabla_V\delta^\triangledown d^\triangledown u-\sum_j\nabla_{[V,D_{e_j}e_j]}u\\
			&=\nabla_V\delta^\triangledown d^\triangledown u-\sum_j\nabla_{D_VD_{e_j}e_j}u,
		\end{align*}
	$$-2\sum_j\nabla_{[e_j,V]}\nabla_{e_j}u=-2f_v\delta^\triangledown d^\triangledown u,$$
	and
	\begin{align*}
			-\sum_j[e_j,[e_j,V]]=&-\sum_jD_{e_j}[e_j,V]\\
			=&V+\sum_jD_{e_j}D_Ve_j\\
			=&V+\sum_j \left( R(e_j,V)e_j+D_VD_{e_j}e_j \right) \\
			=&(2-n)V+\sum_jD_VD_{e_j}e_j,
		\end{align*}
	where we use $\sum_jR(e_j,X)e_j=(1-n)X$ on $S^n$. Hence
	$$\delta^\triangledown d^\triangledown\nabla_Vu=\delta^\triangledown R^\triangledown(V)u-2\mathfrak{R}^\triangledown(d^\triangledown u)(V)+\nabla_V\delta^\triangledown d^\triangledown u+(2-n)\nabla_Vu-2f_v\delta^\triangledown d^\triangledown u.$$
	By equation (\ref{YMH}), at $x$ we have
	\begin{align*}
		&\sum_k\langle \nabla_V(\delta^\triangledown R^\triangledown(e_k)),R^\triangledown(V,e_k)\rangle   \\
		=&\sum_k\langle \frac12\nabla_V(u\otimes(\nabla_{e_k}u)^\ast-\nabla_{e_k}u\otimes u^\ast),R^\triangledown(V,e_k)\rangle \\
		=&\sum_k\langle \nabla_Vu,R^\triangledown(V,e_k)\nabla_{e_k}u\rangle +\langle u,R^\triangledown(V,e_k)\cdot\nabla_V\nabla_{e_k}u\rangle \\
		=&-\langle i_VR^\triangledown,\frac12(d^\triangledown u\otimes (\nabla_Vu)^\ast-\nabla_Vu\otimes(d^\triangledown u)^\ast)\rangle -\sum_k\langle R^\triangledown(V,e_k)u,\nabla_V\nabla_{e_k}u\rangle .
	\end{align*}
	And
	\begin{displaymath}
		\sum_k\langle \nabla_{e_k}(\delta^\triangledown R^\triangledown(V)),R^\triangledown(V,e_k)\rangle =\langle i_VR^\triangledown,\frac12(d^\triangledown u\otimes (\nabla_Vu)^\ast-\nabla_Vu\otimes(d^\triangledown u)^\ast)\rangle -\sum_k\langle R^\triangledown(V,e_k)u,\nabla_{e_k}\nabla_Vu\rangle .
	\end{displaymath}	
	Note that at $x$, by the equation (\ref{YMH}) we have $$\sum_k\langle \nabla_{[V,e_k]}u,R^\triangledown(V,e_k)u\rangle =\sum_kf_v\langle u,R^\triangledown(e_k,V)\nabla_{e_k}u\rangle =-f_v\langle i_VR^\triangledown,\delta^\triangledown R^\triangledown\rangle .$$
	Hence by $R^\triangledown(V,e_k)u=\nabla_V\nabla_{e_k}u-\nabla_{e_k}\nabla_Vu-\nabla_{[V,e_k]}u$, we have
	\begin{align*}
		&\langle i_V\Delta^\triangledown R^\triangledown,i_VR^\triangledown\rangle \\
		=&\sum_k\langle \nabla_V(\delta^\triangledown R^\triangledown(e_k))-\delta^\triangledown R^\triangledown(D_Ve_k)-\nabla_{e_k}(\delta^\triangledown R^\triangledown(V))+\delta^\triangledown R^\triangledown(D_{e_k}V),R^\triangledown(V,e_k)\rangle \\
		=&-\langle i_VR^\triangledown,(d^\triangledown u\otimes(\nabla_Vu)^\ast-\nabla_Vu\otimes(d^\triangledown u)^\ast)\rangle -\langle i_VR^\triangledown,\frac12((i_VR^\triangledown\cdot u)\otimes u^\ast-u\otimes(i_VR^\triangledown\cdot u)^\ast)\rangle .
	\end{align*}
	And thus $$\langle \mathscr{S}_1(i_VR^\triangledown,\nabla_Vu),i_VR^\triangledown\rangle =(4-n)|i_VR^\triangledown|^2-2f_v\langle \delta^\triangledown R^\triangledown,i_VR^\triangledown\rangle .$$
	On the other hand, by equation (\ref{YMH}) and note that $-2\mathfrak{R}^\triangledown(d^\triangledown u)(V)-2i_VR^\triangledown\llcorner d^\triangledown u=0$, we have
	$$\mathscr{S}_2(i_VR^\triangledown,\nabla_Vu)=(2-n)\nabla_Vu-2f_v\delta^\triangledown d^\triangledown u.$$
	Then we finish the proof.
\end{proof}

\begin{yl}\label{trace}
	Let $\{v_i\mid1\le i\le n+1\}\in\mathbb{R}^{n+1}$ be an orthogonal vector in $\mathbb{R}^{n+1}$ and the corresponding Killing field is $V_i$. Then
	\begin{equation}
		\sum_{i=1}^{n+1}\mathscr{L}^{(\nabla,u)}(i_{V_i}R^\triangledown,\nabla_{V_i}u)=2(4-n)\int_{S^n}|R^\triangledown|^2dx+(2-n)\int_M|d^\triangledown u|^2dV.
	\end{equation}
\end{yl}
\begin{proof}
	Assume $v_i$ is the vector in $\mathbb{R}^{n+1}$ such that the $ i $-th component is $1$ and the others are $0$, then at $x=(x^1,...,x^{n+1})$, we have $f_{v_i}(x)=x^i$ and \[\sum_if_{v_i}(x)V_i=\sum_ix^i\frac{\partial}{\partial x^i}-\sum_{i,j}(x^i)^2x^j\frac{\partial}{\partial x^j}=0. \]
	Then according to Lemma \ref{L-VV}, 
	\begin{equation*}
        \sum\limits_i\mathscr{L}^{(\nabla,u)}(i_{V_i}R^\triangledown,\nabla_{V_i} u)=\sum_i \int_{S^n}q\left( v_i,v_i \right) dV
	\end{equation*}
where at any $x\in S^n$, $q$ is a quadratic form on $\mathbb{R}^{ n+1  }$ defined by 
	\begin{equation}\label{q}
		\begin{split}
			q(v,w)=(4-n)\langle  i_VR^\triangledown,i_WR^\triangledown \rangle  +(2-n)\langle  \nabla_Vu,\nabla_Wu \rangle 
		\end{split}
	\end{equation}
	and $V,W\in T S^n$ with respect to $v,w\in \mathbb{R}^{ n+1 }$ defined in Proposition \ref{Killing}. 
	
	Now, we compute the value $q\left( v,w \right) $ at $x\in S^n $. Since $q$ is a quadratic form on $\mathbb{R}^{n+1} $, the trace $\sum\limits_{i=1} q\left( v_i,v_i \right) $ is independent of the choice of  basis of $\mathbb{R}^{ n+1 }$. We assume that $\left\{  e_j \right\}_{ j=1 }^n $ is any orthonormal basis for $T_x S^n $. Then $\left\{  e_0=x ,  e_1,e_2,\cdots, e_n \right\}$ forms an orthonormal basis for $\mathbb{R}^{n+1} $. In particular, at $x\in S^n $, we have
	\begin{align*}
		\sum_{i=1}^{n+1}q(v_i,v_i)=\sum_{j=0}^{ n }q(e_j,e_j).
	\end{align*}
From Proposition \ref{Killing}, at $x$, the Killing vector fields $\varepsilon_0,\varepsilon_1,\cdots,\varepsilon_n$ with respect to $e_0,e_1,e_2,\cdots,e_n $ are 
\[\varepsilon _0=0,\varepsilon _1=e_1,\cdots,\varepsilon _n=e_n .\]

Since $\left\{  e_j \right\}_{ j=1 }^n $ forms an orthonormal basis for $T_x S^n $, we have
\[\sum_{j=1}^n \left| i_{ e_j }R^\nabla \right|=2\left| R^\nabla \right|^2.\] 
Hence, at $x$,
	\begin{align*}
	\sum_{i=1}^{n+1}q(v_i,v_i)=\sum_{j=1}^{ n }q(e_j,e_j)=\left( 4-n \right) \sum_{ j=1 }^n \left| i_{e_j} R^\nabla\right|^2 +\left( 2-n \right)\sum_{ j=1 }^n \left| \nabla_{e_j} u\right|^2 =2\left( 4-n \right) \left| R^\nabla \right|^2 +\left( 2-n \right)\left| d^\nabla u \right|^2 
\end{align*}
and we complete the proof of this lemma.
\end{proof}

From the lemma above, we can immediately prove Theorem \ref{mainthm1} on $S^n (n \geq 5)$.

\begin{thm}
	Assume $(\nabla,u)$ is a Yang-Mills-Higgs pair on $S^n$ for $n\ge 5$, then $(\nabla,u)$ is weakly stable if and only if  $R^\triangledown=0$, $d^\triangledown u=0$ and $|u|=1$.
\end{thm}
\begin{proof}
	If $(\nabla,u)$ satisfies $R^\triangledown=0$, $d^\triangledown u=0$ and $|u|=1$, then $\mathscr{A}(\nabla,u)=0$ is a minimum of $\mathscr{A}$ and $(\nabla,u)$ is obviously weakly stable. 
	
	On the other hand, assume $(\nabla,u)$ is weakly stable, then for any $v\in\mathbb{R}^{n+1}$, $\mathscr{L}^{(\nabla,u)}(i_VR^\triangledown,\nabla_Vu)\ge0$. By Lemma \ref{trace}, we have $R^\triangledown=0$ and $d^\triangledown u=0$. From Equation \eqref{YMH}, we have $u=0$ or $|u|=1$. However, $\left( \nabla,0 \right) $	can not be weakly stable due to the expression of $\mathscr{A}$. In fact, we choose a nonzero section $w\in\Omega^0(E)$ such that $d^\triangledown w=0$ and perturb $\left( \nabla,u \right)$ along $\left( 0,w \right) \in \mathcal{C}$. Then 
	\[\mathscr{L}^{ \left( \nabla,0 \right)  }\left( 0,w \right) =-\frac\lambda2\int_{S^n}|w|^2dV<0\] 
	which contradicts with the weakly stable condition.  Thus $|u|=1$.
\end{proof}

The case when $n=4$ is similar as the proof above. First, we can get $ d^\triangledown u=0$. Then the Yang-Mills condition can be obtained according to Equation \eqref{YMH}.

\section{The Higgs fields of $\Omega^0(\mathfrak{g}_E)$}
In this section, we consider the Yang-Mills-Higgs functional 
\begin{equation}\label{functional2}
	\mathscr{A}(\nabla,\Phi)=\frac12\int_M|R^\triangledown|^2+|d^\triangledown \Phi|^2+\frac\lambda4(1-|\Phi|^2)^2dV,
\end{equation}
 where $\Phi\in\Omega^0(\mathfrak{g}_E)$. Assume that $(\nabla,\Phi)$ is a Yang-Mills-Higgs pair satisfying the equation (\ref{YMH2}) and $(\nabla^t,\Phi^t)$ is a curve on \[\tilde{\mathcal{A}}=\{(\tilde\nabla,\tilde\Phi)\mid \tilde\nabla \textrm{ is a connection of }E,\ \tilde\Phi\in\Omega^0(\mathfrak{g}_E)\}\] such that $\nabla^0=\nabla$ and $\Phi^0=\Phi$. If we assume $\frac{d}{dt}(\nabla^t,\Phi^t)\mid_{t=0}=(B,\phi)$, then the second variation of $\mathscr{A}$ is
\begin{equation}
	\begin{split}
		\left.\frac{d^2}{dt^2}\mathscr{A}(\nabla^t,\Phi^t)\right|_{t=0}=\int_M&\langle \delta^\triangledown d^\triangledown B+\mathfrak{R}^\triangledown(B)-[[B,\Phi],\Phi],B\rangle+2\langle  d^\triangledown\Phi,[B,\phi]\rangle\\
		&+2\langle [B,\Phi],d^\triangledown\phi\rangle+\langle \delta^\triangledown d^\triangledown\phi+\lambda\langle \Phi,\phi\rangle\Phi-\frac\lambda2(1-|\Phi|^2)\phi,\phi\rangle dV.
	\end{split}
\end{equation}
For any gauge transformation $g\in\mathcal{G}$, $g$ acts on $(\nabla,\Phi)$ such that
$$g\cdot(\nabla,\Phi)=(\nabla^g,\Phi^g)=(g\circ\nabla\circ g^{-1},g\circ\Phi\circ g^{-1}).$$
Then for any $\sigma\in\Omega^0(\mathfrak{g}_E)$, assume $g_t=exp(t\sigma)$ is a family of gauge transformations. The variation direction of $(\nabla,u)$ along $\sigma$ is
\begin{align*}
	\frac{d}{dt}(\nabla^{g_t},\Phi^{g_t})\mid_{t=0}=(-d^\triangledown\sigma,[\sigma,\Phi]).
\end{align*}
Similar to the case when the Higgs fields take values in $\Omega^0(E)$, define
\begin{align*}
	\zeta:\Omega^0(\mathfrak{g}_E)&\to\Omega^1(\mathfrak{g}_E)\times\Omega^0(\mathfrak{g}_E),\\
	\sigma&\mapsto(-d^\triangledown\sigma,[\sigma,\Phi]),
\end{align*}
and $\mathcal{C}=Im(\zeta)^\perp$. By direct calculation, we have
\begin{equation}
	\mathcal{C}=\{(B,\phi)\in\Omega^1(\mathfrak{g}_E)\times\Omega^0(\mathfrak{g}_E)\mid\delta^\triangledown B=[\Phi,\phi]\}.
\end{equation}
On $\mathcal{C}$, we have 
\begin{align*}
	2\int_M\langle  [B,\Phi],d^\triangledown\phi \rangle  dV&=2\int_M\langle  [\Phi,d^\nabla \phi],B \rangle  dV=2\int_M\langle  d^\nabla \left[ \Phi,\phi \right]-\left[ d^\nabla \Phi ,\phi \right],B \rangle  dV\\
	&=2\int_M\langle  \delta^\triangledown B,[\Phi,\phi] \rangle  -\langle  B, [d^\triangledown\Phi,\phi] \rangle  dV\\
	&=\int_M\langle  [[\Phi,\phi],\Phi],\phi \rangle  +\langle  d^\triangledown\delta^\triangledown B,B \rangle  -2\langle  B,[d^\triangledown\Phi,\phi] \rangle  dV.
\end{align*}
Then if $\left( B,\phi \right) \in \mathcal{C}$, the second variation formula of $\mathscr{A}$ can be written as
\begin{equation}
	\begin{split}
		\left.\frac{d^2}{dt^2}\mathscr{A}(\nabla^t,\Phi^t)\right|_{t=0}=\int_M&\langle  \Delta^\triangledown B+\mathfrak{R}^\triangledown(B)-[[B,\Phi],\Phi]-2[d^\triangledown\Phi,\phi],B \rangle  \\
		&+\langle  \delta^\triangledown d^\triangledown \phi+[[\Phi,\phi],\Phi]+2d^\triangledown\Phi\llcorner B+\lambda\langle  \Phi,\phi \rangle  \Phi-\frac\lambda2(1-|\Phi|^2)\phi,\phi \rangle  dV,
	\end{split}
\end{equation}
where $d^\triangledown\Phi\llcorner B=\sum\limits_{i=1}^n[\nabla_{e_i}\Phi,B(e_i)]$. Hence, we can define an operator \[\mathscr{S}^{(\nabla,\Phi)}: \Omega^1(\mathfrak{g}_E)\times\Omega^0(\mathfrak{g}_E)\to\Omega^1(\mathfrak{g}_E)\times\Omega^0(\mathfrak{g}_E)\] 
where 
\begin{equation}
	\begin{split}
		\mathscr{S}^{(\nabla,\Phi)}(B,\phi)=(&\Delta^\triangledown B+\mathfrak{R}^\triangledown(B)-[[B,\Phi],\Phi]-2[d^\triangledown\Phi,\phi],\\
		&\delta^\triangledown d^\triangledown\phi+[[\Phi,\phi],\Phi]+2d^\triangledown\Phi\llcorner B+\lambda\langle \Phi,\phi\rangle \Phi-\frac\lambda2(1-|\Phi|^2)\phi).
	\end{split}
\end{equation}
\begin{yl}
	$\mathscr{S}^{(\nabla,\Phi)}(\mathcal{C})\subset\mathcal{C}$.
\end{yl}
\begin{proof}
	Denote $\mathscr{S}^{(\nabla,\Phi)}=(\mathscr{S}_1,\mathscr{S}_2)$. Note that $\mathscr{S}^{(\nabla,\Phi)}$ is self-adjoint on $\Omega^1(\mathfrak{g}_E)\times\Omega^0(\mathfrak{g}_E)$, we only need to prove that for any $(B,\phi)\in\mathcal{C}=\operatorname{Im}^\perp \left( \zeta \right) $ and $\sigma\in\Omega^0(\mathfrak{g}_E)$, we have
	$$\int_M\langle  B,\mathscr{S}_1(\zeta(\sigma)) \rangle  +\langle  \phi,\mathscr{S}_2(\zeta(\sigma)) \rangle  dV=0.$$
	Since $\zeta\left( \sigma \right)=\left( -d^\nabla \sigma,\left[ \sigma,\Phi  \right] \right) $, we have
	\begin{equation*}
		\begin{split}
			\mathscr{S}_1 \left( \zeta\left( \sigma \right)\right) &=-\triangle^\nabla \left( d^\nabla \sigma  \right)-\mathfrak{R}^\nabla \left( d^\nabla \sigma \right) +\left[ \left[ d^\nabla\sigma ,\Phi \right],\Phi \right] -2\left[ d^\nabla\Phi,\left[ \sigma,\Phi \right] \right],\\
			\mathscr{S}_2 \left( \zeta\left( \sigma \right)\right)&=\delta^\nabla d^\nabla \left[ \sigma,\Phi \right]+\left[ \left[ \Phi,\left[ \sigma,\Phi \right] \right],\Phi \right]-2\sum_{i=1}^{n} \left[ \nabla_{e_i} \Phi,\nabla_{e_i}  \sigma\right]+\lambda \langle   \Phi,\left[ \sigma,\Phi \right] \rangle\Phi -\frac{\lambda}{2}\left( 1-\left| \Phi \right|^2 \right) \left[ \sigma,\Phi \right]  \\
			&=\delta^\nabla d^\nabla \left[ \sigma,\Phi \right]+\left[ \left[ \Phi,\left[ \sigma,\Phi \right] \right],\Phi \right]-2d^\triangledown\Phi\llcorner d^\triangledown\sigma -\frac{\lambda}{2}\left( 1-\left| \Phi \right|^2 \right) \left[ \sigma,\Phi \right] 
		\end{split}
	\end{equation*}
	where we use the fact that $\langle  \Phi,\left[ \sigma,\Phi \right] \rangle =0$ in the last equality.  Note that
	\[\delta^\nabla d^\triangledown d^\triangledown\sigma=\delta^\nabla [R^\triangledown,\sigma]=[\delta^\triangledown R^\triangledown,\sigma]-\mathfrak{R}^\triangledown(d^\triangledown\sigma),\] we have 
	$$\Delta^\triangledown d^\triangledown\sigma+\mathfrak{R}^\triangledown(d^\triangledown\sigma)=d^\triangledown\delta^\triangledown d^\triangledown\sigma+[\delta^\triangledown R^\triangledown,\sigma].$$
	For the third term in $\mathscr{S}_1\left( \zeta \left( \sigma \right)  \right) $, we have
	\begin{align*}
		[[d^\triangledown\sigma,\Phi],\Phi]&=[d^\triangledown[\sigma,\Phi],\Phi]-[[\sigma,d^\triangledown\Phi],\Phi]\\
		&=d^\triangledown[[\sigma,\Phi],\Phi]-[[\sigma,\Phi],d^\triangledown\Phi]-[[\sigma,d^\triangledown\Phi],\Phi]\\
		&=d^\triangledown[[\sigma,\Phi],\Phi]+2[d^\triangledown\Phi,[\sigma,\Phi]]+[[d^\triangledown\Phi,\Phi],\sigma]\\ 
		&=d^\triangledown[[\sigma,\Phi],\Phi]+2[d^\triangledown\Phi,[\sigma,\Phi]]+[\delta^\nabla R^\nabla ,\sigma],
	\end{align*}
	where we use the Jacobi identity in the third equality to $\left[ \left[ \sigma,d^\nabla \Phi \right],\Phi \right]$ and apply Equation \eqref{YMH2} in the last equality. Combining the equations above, we have \[\mathscr{S}_1(\zeta(\sigma))=d^\triangledown(-\delta^\triangledown d^\triangledown\sigma+[[\sigma,\Phi],\Phi]).\] 
	
	In the following we deal with $\mathscr{S}_2(\zeta(\sigma))$. By direct computation, we have
	\[\delta^\triangledown d^\triangledown[\sigma,\Phi ]=[\sigma,\delta^\triangledown d^\triangledown\Phi]+2d^\triangledown\Phi\llcorner d^\triangledown\sigma+[\delta^\triangledown d^\triangledown\sigma,\Phi] .\]
	Inserting it to $\mathscr{S}_2(\zeta(\sigma))$ and applying the second equation in Equation \eqref{YMH2}, we obtain
	$$ \mathscr{S}_2(\zeta(\sigma))=[\delta^\triangledown d^\triangledown\sigma+[[\Phi,\sigma],\Phi],\Phi]. $$
	
	Hence,
	\begin{equation*}
		\begin{split}
			&\int_M\langle  B,\mathscr{S}_1(\zeta(\sigma)) \rangle  +\langle  \phi,\mathscr{S}_2(\zeta(\sigma)) \rangle  dV\\
			=&\int_M\langle  B,d^\triangledown(-\delta^\triangledown d^\triangledown\sigma+[[\sigma,\Phi],\Phi]) \rangle  +\langle  \phi,[\delta^\triangledown d^\triangledown\sigma+[[\Phi,\sigma],\Phi],\Phi] \rangle  dV\\
			=&\int_M\langle  [\Phi,\phi],-\delta^\triangledown d^\triangledown\sigma+[[\sigma,\Phi],\Phi] \rangle  +\langle  \phi,[\delta^\triangledown d^\triangledown\sigma+[[\Phi,\sigma],\Phi],\Phi] \rangle  dV=0,
		\end{split}
	\end{equation*}
	where we use the condition $\left( B,\Phi \right)\in \mathcal{C}$, i.e. $\delta^\nabla B=\left[ \Phi,\phi \right]$ in the second equality.
\end{proof}
Hence $\mathscr{S}^{(\nabla,\Phi)}\mid_{\mathcal{C}}:\mathcal{C}\to\mathcal{C}$ is an elliptic self-adjoint operator. If we assume the minimum eigenvalue of $\mathscr{S}^{(\nabla,u)}\mid_{\mathcal{C}}$ is $\lambda_1$, then $(\nabla,u)$ is called weakly stable if $\lambda_1\ge0$ and stable if $\lambda_1>0$. For any $(B,\phi)\in\mathcal{C}$, define
\begin{align}
	\mathscr{L}^{(\nabla,\Phi)}(B,\phi)=\int_M\langle \mathscr{S}^{(\nabla,\Phi)}(B,\phi),(B,\phi)\rangle dV,
\end{align}
Then $\frac{d^2}{dt^2}\mathscr{A}(\nabla^t,u^t)\mid_{t=0}=\mathscr{L}^{(\nabla,u)}(B,\phi)$.

Now assume $M=S^n$ and $(\nabla,\Phi)$ is a Yang-Mills-Higgs pairs satisfying (\ref{YMH2}). For any $v\in\mathbb{R}^{n+1}$, let $f_v(x)=v\cdot x\in C^\infty(S^n)$ and $V=grad(f_v)$ be the Killing field (\ref{Vector}) on $S^n$. Similar to the case when the Higgs fields in $\Omega^0(E)$, assume $B_v=i_VR^\triangledown$ and $\phi_v=\nabla_V\Phi$. Then $(B_v,\phi_v)\in\mathcal{C}$. Similarly, we have
\begin{yl}
	Assume $(\nabla,\Phi)$ is a Yang-Mills-Higgs pair on $S^n$. Then for any $v\in\mathbb{R}^{n+1}$, we have
	\begin{equation}
		\mathscr{L}^{(\nabla,\Phi)}(i_VR^\triangledown,\nabla_V\Phi)=\int_{S^n}(4-n)|i_VR^\triangledown|^2+(2-n)|\nabla_V\Phi|^2-2f_v(\langle  \delta^\triangledown R^\triangledown,i_VR^\triangledown \rangle  +\langle  \delta^\triangledown d^\triangledown\Phi,\nabla_V\Phi \rangle  )dV.
	\end{equation}
\end{yl}
\begin{proof}
	Similar to the proof of Lemma \ref{L-VV}, we have 
	$$\Delta^\triangledown i_VR^\triangledown+\mathfrak{R}^\triangledown(i_VR^\triangledown)=i_V\Delta^\triangledown R^\triangledown+(4-n)i_VR^\triangledown-2f_v\delta^\triangledown R^\triangledown$$
	and
	$$\delta^\triangledown d^\triangledown\nabla_V\Phi=[\delta^\triangledown R^\triangledown(V),\Phi]-2\mathfrak{R}^\triangledown(d^\triangledown\Phi)(V)+\nabla_V\delta^\triangledown d^\triangledown\Phi+(2-n)\nabla_V\Phi-2f_v\delta^\triangledown d^\triangledown\Phi.$$
	where $\mathfrak{R}^\triangledown(d^\triangledown \Phi)(V)=\sum_j[R^\triangledown(e_j,V),\nabla_{e_j}\Phi]$ and $\{e_1,...,e_n\}$ is an orthogonal basis of $TS^n$. Furthermore, by the Equation \eqref{YMH2}, \[[\delta^\triangledown R^\triangledown(V),\Phi]=[[\nabla_V\Phi,\Phi],\Phi]\text{ and } \nabla_V\delta^\triangledown d^\triangledown\Phi=-\lambda\langle \nabla_V\Phi,\Phi\rangle \Phi+\frac\lambda2(1-|\Phi|^2)\nabla_V\Phi. \]
	Also noting that $d^\triangledown\Phi\llcorner i_VR^\triangledown=\mathfrak{R}^\triangledown(d^\triangledown\Phi)(V)$, we have
	\begin{align*}
		\mathscr{L}^{(\nabla,\Phi)}(i_VR^\triangledown ,\nabla_V\Phi)=\int_{S^n}&\left( \langle  i_V\Delta^\triangledown R^\triangledown+(4-n)i_VR^\triangledown-2f_v\delta^\triangledown R^\triangledown-[[i_VR^\triangledown,\Phi],\Phi]-2[d^\triangledown\Phi,\nabla_V\Phi],i_VR^\triangledown \rangle  \right.\\
		&\left.+\langle  (2-n)\nabla_V\Phi-2f_v\delta^\triangledown d^\triangledown\Phi,\nabla_V\Phi \rangle   \right) dV.
	\end{align*}
	By Equation \eqref{YMH2}, we have at $x\in S^n $
	\begin{align*}
		i_V\Delta^\triangledown R^\triangledown(e_k)=&d^\triangledown[d^\triangledown\Phi,\Phi](V,e_k)\\
		=&\nabla_V[\nabla_{e_k}\Phi,\Phi]-\nabla_{e_k}[\nabla_V\Phi,\Phi]-\left[ \nabla_{ \left[ V,e_k \right] } \Phi,\Phi \right]\\
		=&[\nabla_V\nabla_{e_k}\Phi,\Phi]+[\nabla_{e_k}\Phi,\nabla_V\Phi]-[\nabla_{e_k}\nabla_V\Phi,\Phi]-[\nabla_V\Phi,\nabla_{e_k}\Phi]-\left[ \nabla_{ \left[ V,e_k \right] } \Phi,\Phi \right]\\
		=&[[R^\triangledown(V,e_k),\Phi],\Phi]+2[\nabla_{e_k}\Phi,\nabla_V\Phi]\\ 
		=&\left[\left[ i_V R^\nabla,\Phi \right] ,\Phi \right]\left( e_k \right) +2\left[ d^\nabla \Phi,\nabla_V \Phi \right]\left( e_k \right) .
	\end{align*}
	Thus the second variation of $\mathscr{A}$ at $\left( \nabla,\Phi \right) $ along $\left( i_V R^\nabla,\nabla_V \Phi \right) $ is $$\mathscr{L}^{(\nabla,\Phi)}(i_VR^\triangledown,\nabla_V\Phi)=\int_{S^n}(4-n)|i_VR^\triangledown|^2+(2-n)|\nabla_V\Phi|^2-2f_v(\langle  \delta^\triangledown R^\triangledown,i_VR^\triangledown \rangle  +\langle  \delta^\triangledown d^\triangledown\Phi,\nabla_V\Phi \rangle  )dV.$$
	We finish the proof.
\end{proof}
For an orthogonal basis $\{v_i\mid 1\le i\le n+1\}$ of $\mathbb{R}^{n+1}$, we assume $\{V_i\}$ is the corresponding Killing fields. We have proved that $\sum_if_{v_i}V_i=0$. Hence by adding $\mathscr{L}^{(\nabla,\Phi)}(i_{V_i}R^\triangledown,\nabla_{V_i}\Phi)$ together, we can prove the following lemma.
\begin{yl}
	Assume $\{v_i\}$ and $\{V_i\}$ defined as above, then we have
	$$\sum_i\mathscr{L}^{(\nabla,\Phi)}(i_{V_i}R^\triangledown,\nabla_{V_i}\Phi)=\int_{S^n}2(4-n)|R^\triangledown|^2+(2-n)|d^\triangledown \Phi|^2dV.$$
\end{yl}
From the lemma above, we can immediately prove the stability theorem.
\begin{thm}
	Assume $(\nabla,\Phi)$ is a weakly stable Yang-Mills-Higgs pair on $S^n$ for $n\ge5$, then $R^\triangledown=0$, $d^\triangledown \Phi=0$ and $|\Phi|\equiv1$.\\
	And if $n=4$, then $d^\triangledown \Phi=0$, $|\Phi|\equiv1$ and $R^\triangledown$ is a Yang-Mills connection.
\end{thm}
\section{The energy identity for a sequence of Yang-Mills-Higgs pairs}

In this section, we assume $M$ is a four-dimensional compact Riemannian manifold. Let $\{\nabla_i,u_i\}$ be a sequence of Yang-Mills-Higgs pairs satisfying the Yang-Mills-Higgs equation (\ref{YMH}) with uniformly bounded energy $\mathscr{A}(\nabla_i,u_i)\le K$.
\subsection{There is no energy concentration point for the Higgs field.}
Assume $(\nabla,u)$ satisfies the equation (\ref{YMH}). In this part, we will give the estimate of $\|u\|_{L^\infty}$ and $\|d^\triangledown u\|_{L^2}$. In fact, we will show that if a sequence of Yang-Mills-Higgs $\{(\nabla_i,u_i)\}$ weakly converges to $(\nabla,u)$ in $W^{1,2}(\mathfrak{g}_E)\times W^{1,2}(E)$, then $u_i$ converges to $u$ smoothly.
\begin{yl}\label{L^infty}
	Assume $(\nabla,u)$ is a Yang-Mills-Higgs pair, then $\|u\|_{L^\infty}\le1$.
\end{yl}
\begin{proof}
	Assume $|u|^2$ attains the maximum at $x_0\in M$.
	Then at $x_0$, we have
	\begin{displaymath}
		0\le\Delta|u|^2(x_0)=(2\langle \delta^\triangledown d^\triangledown u,u\rangle -2|d^\triangledown u|^2)(x_0)\le\lambda(1-|u|^2)|u|^2(x_0).
	\end{displaymath}
	Hence $|u|^2(x_0)\le1$ and this implies the lemma.
\end{proof}
 \begin{yl}\label{concentration}
	Assume $(\nabla,u)$ is a Yang-Mills-Higgs pair and $\mathscr{A}(\nabla,u)\le K$. Then for any $B_r=B_r(x_0)\subset M$, we have 
	\begin{equation}
		\|d^\triangledown u\|_{L^2(B_r)}^2\le C(r+r^4),
	\end{equation}
	where $C=C(K,\lambda)$.
\end{yl}
\begin{proof}
	Choose a cut off function $\eta$ such that $\eta\mid_{B_r}\equiv1$, $supp(\eta)\subset B_{2r}$ and $|d\eta|\le Cr^{-1}$. Then by the equation (\ref{YMH}), we have
	\begin{align*}
		\int_{B_r}|d^\triangledown u|^2dV&\le\int_{B_{2r}}\langle d^\triangledown u,\eta\cdot d^\triangledown u\rangle dV\\
		&=\int_{B_{2r}}\langle u,\delta^\triangledown(\eta\cdot d^\triangledown u)\rangle dV\\
		&=\int_{B_{2r}}\langle u,\eta\cdot\frac\lambda2(1-|u|^2)u+d\eta\# d^\triangledown u\rangle dV\\
		&\le Cr^4+C\|d\eta\|_{L^2(B_{2r})}\|d^\triangledown u\|_{L^2(B_{2r})}\\
		&\le C(r+r^4).
	\end{align*}
\end{proof}
The above lemmas imply the following theorem immediately.
\begin{thm}\label{Higgs bound}
	Assume $(\nabla,u)$ is a Yang-Mills-Higgs pair and $\mathscr{A}(\nabla,u)\le K$, then for any $x\in M$, we have
	\begin{align*}
		\int_{B_r(x)}|d^\triangledown u|^2+\frac\lambda4(1-|u|^2)^2dV\le C(K,\lambda)(r+r^4).
	\end{align*}
    In particular, we have
    \begin{align*}
    	\lim_{r\to0}\int_{B_r(x)}|d^\triangledown u|^2+\frac\lambda4(1-|u|^2)^2dV=0.
    \end{align*}
\end{thm}
\subsection{The $\epsilon$-regularity}
The $\epsilon$-regularity theorem is proved by Uhlenbeck \cite{KK1} for Yang-Mills connections, Struwe \cite{struwe} for Yang-Mills flows and Hong-Fang \cite{HF} for Yang-Mills-Higgs flows. For completeness, we first give the proof of the $\epsilon$-regularity of Yang-Mills-Higgs pairs here. Let $i(M)$ be the injective radius of $M$. Then for any $r<i(M)$ and $x\in M$, there is a trivialization $\nabla=d+A$ in $B_r(x)$.
\begin{yl}[$\epsilon$-regularity]\label{regularity}
	Assume $(\nabla,u)$ is a Yang-Mills-Higgs pair. There exists $\epsilon_0=\epsilon_0(M)$ such that if for some $R<i(M)$ and $x_0\in M$,
	\begin{displaymath}
		\int_{B_R(x_0)}|R^\triangledown|^2+|d^\triangledown u|^2dV\le\epsilon_0,
	\end{displaymath}
	then for any $r_1<R$, we have
	\begin{displaymath}
		\sup_{B_{\frac{r_1}2}(x_0)}(|R^\triangledown|^2+|d^\triangledown u|^2)\le Cr_1^{-4}\int_{B_{r_1}(x_0)}|R^\triangledown|^2+|d^\triangledown u|^2dV,
	\end{displaymath}
    where $C=C(M)$.
\end{yl}
\begin{proof}
	There exists $r_0<r_1$, such that
	\begin{align*}
		(r_1-r_0)^4\sup_{D_{r_0}(x_0)}(|R^\triangledown|^2+|d^\triangledown u|^2)=\sup_{0\le r\le r_1}((r_1-r)^4\sup_{D_r(x_0)}(|R^\triangledown|^2+|d^\triangledown u|^2))
	\end{align*}
	where $D_r \left( x \right)=\partial B_r\left( x \right) $. Then there exists $x_1\in D_{r_0}(x_0)$ such that
	\begin{align*}
		(|R^\triangledown|^2+|d^\triangledown u|^2)(x_1)=\sup_{D_{r_0}(x_0)}(|R^\triangledown|^2+|d^\triangledown u|^2).
	\end{align*}
	Define $e_0=(|R^\triangledown|^2+|d^\triangledown u|^2)(x_1)$. We claim that $e_0\le16(r_1-r_0)^{-4}$. In fact, if we suppose $e_0>16(r_1-r_0)^{-4}$, then $\rho_0=e_0^{-\frac14}<\frac{r_1-r_0}2$. Define
	\begin{align*}
		\rho:B_1(0)\to B_{\rho_0}(x_1)
	\end{align*}
	such that $\rho(x)=x_1+\rho_0 x$, and 
	\begin{align*}
		&\tilde A(x)=\rho^\ast A(x)=\rho_0A(x_1+\rho_0x),\\
		&\tilde u(x)=\rho^\ast u(x)=u(x_1+\rho_0 x).
	\end{align*}
	Let $\widetilde\nabla=d+\tilde A$ be a connection of $\rho^\ast E$ over $B_1(0)$ and $e_{\rho_0}=|R^{\tilde\triangledown}|^2+\rho_0^2|d^{\tilde\triangledown}\tilde u|^2=\rho_0^4(|R^\triangledown|^2+|d^\triangledown u|^2)$. Note that $B_{\rho_0}(x_1)\subset B_{\frac{r_0+r_1}2}(x_0)$, we have
	\begin{align*}
		1=e_{\rho_0}(0)&\le\sup_{B_1(0)}e_{\rho_0}
		\\
		&=\rho_0^4\sup_{B_{\rho_0}(x_1)}(|R^\triangledown|^2+|d^\triangledown u|^2)\\
		&\le\rho_0^4\sup_{B_{\frac{r_0+r_1}2}(x_0)}(|R^\triangledown|^2+|d^\triangledown u|^2)\\
		&=\rho_0^4(\frac{r_1-r_0}2)^{-4}(r_1-\frac{r_0+r_1}2)^4\sup_{B_{\frac{r_0+r_1}2}(x_0)}(|R^\triangledown|^2+|d^\triangledown u|^2)\\
		&\le\rho_0^4(\frac{r_1-r_0}2)^{-4}(r_1-r_0)^4e_0\\
		&=16,
	\end{align*}
	which implies $|R^{\tilde\triangledown}|^2+\rho_0^2|d^{\tilde\triangledown}\tilde u|^2\le 16$ on $B_1(0)$. By equation (\ref{YMH}) and the Bochner formula, we have
	\begin{align*}
		\nabla^\ast\nabla R^\triangledown&=d^\triangledown \delta^\triangledown R^\triangledown-R^\triangledown\circ(Ric\wedge I+2R_M)+R^\triangledown\# R^\triangledown\\
		&=d^\triangledown(d^\triangledown u\#u)-R^\triangledown\circ(Ric\wedge I+2R_M)+R^\triangledown\# R^\triangledown\\
		&=R^\triangledown\#R^\triangledown\#u+d^\triangledown u\#d^\triangledown u-R^\triangledown(Ric\wedge I+2R_M)+R^\triangledown\#R^\triangledown,
	\end{align*}
	and
	\begin{align*}
		\nabla^\ast\nabla d^\triangledown u&=d^\triangledown\delta^\triangledown d^\triangledown u+\delta^\triangledown d^\triangledown d^\triangledown u-d^\triangledown u\circ Ric+R^\triangledown\#d^\triangledown u\\
		&=d^\triangledown(\frac\lambda2(1-|u|^2)u)+\delta^\triangledown(R^\triangledown\cdot u)-d^\triangledown u\circ Ric+R^\triangledown\#d^\triangledown u\\
		&=d^\triangledown u\#u\#u+R^\triangledown\#d^\triangledown u-d^\triangledown u\circ Ric.
	\end{align*}
	Hence
	\begin{align*}
		\Delta e_{\rho_0}&=\rho_0^6\Delta(|R^\triangledown|^2+|d^\triangledown u|^2)\\
		&=2\rho_0^6(\langle  \nabla^\ast\nabla R^\triangledown,R^\triangledown \rangle  -|\nabla R^\triangledown|^2+\langle  \nabla^\ast\nabla d^\triangledown u,d^\triangledown u \rangle  -|\nabla d^\triangledown u|^2)\\
		&\le C\rho_0^6(|R^\triangledown|^3+|R^\triangledown||d^\triangledown u|^2+|R^\triangledown|^2+|d^\triangledown u|^2)\\
		&\le C(|R^{\tilde\triangledown}|^3+\rho_0^2|R^{\tilde\triangledown}||d^{\tilde\triangledown}\tilde u|^2+\rho_0^2|R^{\tilde\triangledown}|^2+\rho_0^4|d^{\tilde\triangledown}\tilde u|^2)\\
		&\le Ce_{\rho_0}.
	\end{align*}
	By Harnack inequality and note that $B_{\rho_0}(x_1)\subset B_R(x_0)$, we have
	\begin{align*}
		1=e_{\rho_0}(0)&\le C\int_{B_1(0)}e_{\rho_0}dV\\
		&=C\int_{B_{\rho_0}(x_1)}|R^\triangledown|^2+|d^\triangledown u|^2dV\\
		&\le C\epsilon_0.
	\end{align*}
	It is a contradiction if $\epsilon_0$ small enough and we prove the claim.
	\\Then we have
	\begin{align*}
		\sup_{B_{\frac{3r_1}4}(x_0)}(|R^\triangledown|^2+|d^\triangledown u|^2)&=(\frac{r_1}4)^{-4}(r_1-\frac{3r_1}4)^4\sup_{B_{\frac{3r_1}4}(x_0)}(|R^\triangledown|^2+|d^\triangledown u|^2)\\
		&\le (\frac{r_1}4)^{-4}(r_1-r_0)^4e_0\\
		&\le Cr_1^{-4},
	\end{align*}
    where, more precisely, $C=2^{12}$. For any $x_2\in B_{\frac{r_1}2}(x_0)$, define
    \begin{align*}
    	\hat A(x)&=\frac{r_1}4A(x_2+\frac{r_1}4x),\\
    	\hat u(x)&=u(x_2+\frac{r_1}4x).
    \end{align*}
    Similarly, $\hat e=|R^{\hat\triangledown}|^2+(\frac{r_1}4)^2|d^{\hat \triangledown}\hat u|^2=(\frac{r_1}4)^4(|R^\triangledown|^2+|d^\triangledown u|^2)$ satisfies $\hat e\le C$ and hence $\Delta\hat e\le C\hat e$ in $B_1(0)$. By Harnack inequality we have
    \begin{align*}
    	\hat e(0)\le C\int_{B_1(0)}\hat e(x)dV=C'\int_{B_{\frac{r_1}4}(x_2)}|R^\triangledown|^2+|d^\triangledown u|^2dV\le C\int_{B_{r_1}(x_0)}|R^\triangledown|^2+|d^\triangledown u|^2dV.
    \end{align*}
    Then we prove the lemma.
\end{proof}
In order to obtain the $\epsilon$-regularity of high order derivatives, we need the following lemma.
\begin{yl}[\cite{KK1}, Theorem 1.3]\label{Lpcur}
	Assume $\nabla=d+A$ is a connection over $B_1$. There exists $\kappa$ and $c$ such that if $\|R^\triangledown\|_{L^2(B_1)}\le\kappa$, then there exists a gauge $g$ such that for any $p\ge2$, there is $d^\ast A^g=0$ and $\|A^g\|_{W^{1,p}(B_1)}\le c\|R^\triangledown\|_{L^p(B_1)}$, where $\kappa$ and $c$ is only determined by $dim(M)$.
\end{yl}
\begin{remark}
	By letting $p\to\infty$ we have $\|A^g\|_{L^\infty(B_1)}\le c\|R^\triangledown\|_{L^\infty(B_1)}$. This is consist with Theorem 2.7 in \cite{KK2}.
\end{remark}
\begin{yl}[$\epsilon$-regularity of high order derivative]\label{C^kreg}
	Assume $(\nabla,u)$ is a Yang-Mills-Higgs pair. There exists $\epsilon_1=\epsilon_1(M)$ such that if for some $R<i(M)$ and $x_0\in M$,
	\begin{align*}
		\int_{B_R(x_0)}|R^\triangledown|^2+|d^\triangledown u|^2dV\le \epsilon_1,
	\end{align*}
    then there exists a gauge transformation $g$ such that for any $r_2<R$ and $k\ge1$, we have
	\begin{displaymath}
		\sup_{B_{\frac{r_2}2}(x_0)}(|d^k A^g|^2+|d^ku^g|^2)\le C(k,r_2).
	\end{displaymath}
\end{yl}
\begin{proof}
	Choose $\epsilon_1=\min\{\epsilon_0,\kappa^2\}$ and $R_1<R_0$. By lemma \ref{Lpcur}, there exists a gauge transformation $g$, such that $d^\ast A^g=0$. For simplicity, we omit the superscript $g$. Then the equations (\ref{YMH}) are
	\begin{align}\label{YMHloc}
		\begin{split}
			&\Delta A+\Phi(A,u)=0,\\
			&\Delta u+\Psi(A,u)=0,
		\end{split}
	\end{align}
	where $\Delta$ is the covariant Laplacian operator on $M$ and
	\begin{align*}
		\Phi(A,u)&=dA\#A+A\#A\#A+du\#u+A\#u\#u,\\ \Psi(A,u)&=dA\#u+A\#du+A\#A\#u-\frac\lambda2(1-|u|^2)u.
	\end{align*}
	By lemma \ref{Lpcur}, we have $\|A\|_{L^\infty}\le C(R_1)\|R^\triangledown\|_{L^\infty}$ in $B_{\frac{R_1}2}(x_0)$. And by Lemma \ref{regularity}, we have
	\begin{displaymath}
		\sup_{B_{\frac12{R_1}}}(|R^\triangledown|^2+|d^\triangledown u|^2)\le C_1.
	\end{displaymath}
	By $L^p$-estimate, for any $R_2<R_1$ and $p>1$, we have  $\|A\|_{W^{2,p}(B_{\frac12R_2})}+\|u\|_{W^{2,p}(B_{\frac12R_2})}\le C(R_1,R_2,p)$. By Sobolev's embedding theorem, $\|A\|_{C^1(B_{\frac12R_2})}+\|u\|_{C^1(B_{\frac12R_2})}\le C(R_1,R_2)$. Differentiating the equation (\ref{YMHloc}) and repeat the above process, we can prove that for any $k\ge1$ and $R_k<R_{k-1}$, we have $\|A\|_{C^k(B_{\frac12R_2})}+\|u\|_{C^k(B_{\frac12R_2})}\le C(R_1,R_2,...,R_k,k)$. For any $r_2<R$, by choosing $R>R_1>R_2>...>r_2$, the lemma is proved.
\end{proof}
\subsection{The proof of theorem \ref{maintheorem2}}

Assume $\{(\nabla_i,u_i)\}$ is a sequence of Yang-Mills-Higgs pairs with uniformly bounded energy $\mathscr{A}(\nabla_i,u_i)\le K$. Define
\begin{align}
	\Sigma=\{x\in M\mid\lim_{R\to0}\liminf_{i\to\infty}\int_{B_R(x)}|R^{\triangledown_i}|^2dV>\epsilon_1\}.
\end{align}
$\mathscr{A}(\nabla_k,u_k)\le K$ implies that $\Sigma$ is finite and the number of elements in $\Sigma$ is no more than $\frac{K}{\epsilon_1}$. The uniform $W^{1,2}$ bounded of $(\nabla_i,u_i)$ implies there is a subsequence of $\{(\nabla_i,u_i)\}$ which weakly converges to a Yang-Mills-Higgs pair $(\nabla_\infty,u_\infty)$ in $\Omega^1(\mathfrak{g}_E)\times\Omega^0(E)$. And by lemma \ref{C^kreg}, the convergence is smooth in $M\backslash\Sigma$. For any $x_0\in\Sigma$, choose $R_0<i(M)$ such that $B_{R_0}(x_0)\cap\Sigma=\{x_0\}$. Assume $\nabla_i=d+A_i$ on $B_{R_0}(x_0)$. Define
\begin{align*}
	\frac1{(r_i^1)^4}=\sup_{B_{R_0}(x_0)}(|R^{\triangledown_i}|^2+|d^{\triangledown_i}u_i|^2),
\end{align*}
Then $\lim\limits_{i\to\infty}r_i^1=0$. Assume
\begin{align*}
	&\tilde A_{1,i}(x)=\rho_i^\ast A_i(x)=r_i^1A_i(x_0+r_i^1x),\\
	&\tilde u_{1,i}(x)=\rho_i^\ast u_i(x)=u_i(x_0+r_i^1x),
\end{align*}
where $\rho_i(x)=x_0+r_i^1x$ maps $B_{R}(0)$ to $B_{Rr_i^1}(x_0)$ for any $R>0$. The pairs $(\tilde\nabla_{1,i},\tilde u_{1,i})$ satisfy
\begin{align*}
	\begin{split}
		\delta^{\tilde\triangledown_{1,i}}R^{\tilde\triangledown_{1,i}}&=-\frac{(r^1_i)^2}2(d^{\tilde\triangledown_{1,i}}\tilde u_{1,i}\otimes\tilde u_{1,i}^\ast-\tilde u_{1,i}\otimes d^{\tilde\triangledown_{1,i}}\tilde u_{1,i}),\\
		\delta^{\tilde\triangledown_{1,i}}d^{\tilde\triangledown_{1,i}}\tilde u_{1,i}&=\frac{\lambda (r^1_i)^2}2(1-|\tilde u_{1,i}|^2)\tilde u_{1,i}.
	\end{split}
\end{align*}
For any $\delta<R_0$ and $R>0$, $i$ is large enough such that $\delta>4Rr_i^1$. Note that $\|R^{\tilde\triangledown_{1,i}}\|_{L^\infty(B_R(0))}\le 1$. Applying theorem \ref{Higgs bound}, lemma \ref{regularity} and lemma \ref{C^kreg}, there is a subsequence of $(\tilde A_{1,i},\tilde u_{1,i})$ converges to $(\tilde A_{1,\infty,R},\tilde u_{1,\infty,R})$ on $B_R(0)$ smoothly under some gauge transformations $g_i$. For simplicity, we assume that the subsequence is $(\tilde A_{1,i},\tilde u_{1,i})$ itself. By choosing $R_1<R_2<...\to\infty$ and subsequence repeatedly, we may assume $(\tilde A_{1,i}^{g_i},\tilde u_{1,i}^{g_i})$ converges to $(\tilde A_{1,\infty},\tilde u_{1,\infty})$ in $C^\infty(\mathbb{R}^4)$ satisfying
\begin{align*}
	\delta^{\tilde\triangledown_{1,\infty}}R^{\tilde\triangledown_{1,\infty}}&=0,\\
	\delta^{\tilde\triangledown_{1,\infty}}d^{\tilde\triangledown_{1,\infty}}\tilde u_{1,\infty}&=0,
\end{align*} 
which implies $\widetilde\nabla_{1,\infty}$ is a Yang-Mills connection and $d^{\tilde\triangledown_{1,\infty}}\tilde u_{1,\infty}=0$. By Uhlenbeck's removable singularity theorem (see \cite{KK2}, corollary 4.3) , $\tilde\nabla_{1,\infty}$ can be extended to a nontrivial Yang-Mills connection on $S^4$ under some gauge transformation. Define 
$$\mathscr{A}(\nabla,u;\Omega)=\frac12\int_\Omega|R^\triangledown|^2+|d^\triangledown u|^2+\frac\lambda4(1-|u|^2)^2dV.$$
The following lemma gives a sufficient condition for the energy on necks to vanish.
\begin{yl}\label{neck}
Assume $(\nabla_i,u_i)$ are a sequence of Yang-Mills-Higgs pairs and $\lim\limits_{i\to\infty}r_i^1=0$. There exists $\epsilon$, such that for any $R,\delta>0$, if
\begin{align}\label{ring_small}
	\liminf_{i\to\infty}\sup_{r\in(\frac{Rr_i^1}2,2\delta)}\mathscr{A}(\nabla_i,u_i;B_{2r}(x_0)\backslash B_r(x_0))<\epsilon,
\end{align}
then
\begin{align*}
	\lim_{R\to\infty}\lim_{\delta\to0}\liminf_{i\to\infty}\mathscr{A}(\nabla_i,u_i;B_\delta(x_0)\backslash B_{Rr_i^1}(x_0))=0.
\end{align*}
\end{yl}
\begin{proof}
	For simplicity, we assume $x_0=0$. For $l\ge-1$, define
	\begin{align*}
		\mathfrak{U}_l&=\{x\mid 2^{-l-1}\le|x|\le2^{-l}\},\\
		S_l&=\{x\mid |x|=2^{-l}\}.
	\end{align*} 
	Divide $A_i$ into the radius part $A_{i,r}$ and the sphere part $A_{i,\theta}$, that is, $A_i=A_{i,r}+A_{i,\theta}$. Recall that by Bochner's formula, the equation (\ref{YMH}) and $\|u_i\|_{L^\infty}\le 1$, we have
	\begin{align*}
		\Delta|R^{\triangledown_i}|^2\le C(M)(|R^{\triangledown_i}|^2+|R^{\triangledown_i}|^3+|R^{\triangledown_i}||D^{\triangledown_i} u_i|^2)-2|d^{\triangledown_i} R^{\triangledown_i}|^2.
	\end{align*}
	Note that
	\begin{align*}
		\Delta(|R^{\triangledown_i}|^2+h)^{\frac12}&=\frac12(|R^{\triangledown_i}|^2+h)^{-\frac12}\Delta|R^{\triangledown_i}|^2+\frac14\sum_j(|R^{\triangledown_i}|^2+h)^{-\frac32}(e_j|R^{\triangledown_i}|^2)^2
	\end{align*}
	and let $h$ tends to $0$, we have
	\begin{align*}
		\Delta|R^{\triangledown_i}|\le C(|R^{\triangledown_i}|+|R^{\triangledown_i}|^2+|d^{\triangledown_i} u_i|^2).
	\end{align*}
	Consider $\tilde A_{i,l}(x)=2^{-l}A_i(2^{-l}x)$, $\tilde u_{i,l}(x)=u_i(2^{-l}x)$ and assume $\tilde\nabla_{i,l}=d+\tilde A_{i,l}$. Note that
	\begin{align*}
		\Delta|R^{\tilde\triangledown_{i,l}}|&=2^{-4l}\Delta|R^{\triangledown_i}|\\
		&\le2^{-4l}C(|R^{\triangledown_i}|+|R^{\triangledown_i}|^2+|d^{\triangledown_i} u_i|^2)\\
		&\le C(2^{-2l}|R^{\tilde\triangledown_{i,l}}|+|R^{\tilde\triangledown_{i,l}}|^2+2^{-2l}|d^{\tilde\triangledown_{i,l}}\tilde u_{i,l}|^2),
	\end{align*}
	and by Harnack's inequality, we have
	\begin{align*}
		\sup_{\mathfrak{U}_l}|R^{\triangledown_i}|&=2^{2l}\sup_{\mathfrak{U}_0}|R^{\tilde\triangledown_{i,l}}|\\
		&\le 2^{2l}C\int_{\mathfrak{U}_{-1}\cup\mathfrak{U}_0\cup\mathfrak{U}_1}2^{-2l}|R^{\tilde\triangledown_{i,l}}|+|R^{\tilde\triangledown_{i,l}}|^2+2^{-2l}|d^{\tilde\triangledown_{i,l}}\tilde u_{i,l}|^2dV\\
		&=C\int_{\mathfrak{U}_{l-1}\cup\mathfrak{U}_l\cup\mathfrak{U}_{l+1}}2^{2l}|R^{\triangledown_i}|+|R^{\triangledown_i}|^2+2^{2l}|d^{\triangledown_i} u_i|^2dV\\
		&\le2^{2l}C\epsilon
	\end{align*}
	Theorem 2.8 and corollary 2.9 in \cite{KK2} show that
	\begin{yl}\label{conn-cur}
		Assume $\nabla=d+A$ be a connection over $\mathfrak{U}_{0}$, there exists $\gamma$ such that if $\|R^\triangledown\|_{L^\infty(\mathfrak{U}_{0})}\le\gamma$, then under a gauge transformation, we have $\delta^\triangledown A=0$ in $\mathfrak{U}_{-1}$, $\delta_\theta^\triangledown A_\theta=0$ on $S_{-1}$ and $S_0$, and $\int_{|x|=r}A_rdS=0$ for any $1\le r\le2$. Moreover, there is $\|A\|_{L^\infty(\mathfrak{U}_0)}\le C\|R^\triangledown\|_{L^\infty(\mathfrak{U}_0)}$ and $\|A\|_{L^2(\mathfrak{U}_0)}\le C\|R^\triangledown\|_{L^2(\mathfrak{U}_0)}$. 
	\end{yl}
	By choosing $\epsilon$ small enough, we may assume $|R^{\tilde\nabla_{i,l}}|\le\gamma$ on $\mathfrak{U}_{0}$ and hence there exists a gauge transformation $g_{i,l}\in\mathscr{G}(\mathfrak{U}_l)$, if denote $A_{i,l}=A_i^{g_{i,l}}$ and $u_{i,l}=u_i^{g_{i,l}}$, we have
	\begin{align*}
		&(1)\ d^\ast A_{i,l}=0,\\
		&(2)\ A_{i,l,\theta}\mid_{S_l}=A_{i,l-1,\theta}\mid_{S_l},\\
		&(3)\ d_\theta^\ast A_{i,l,\theta}=0 \textrm{ on } S_l \textrm{ and } S_{l+1},\\
		&(4)\ \int_{S_l}A_{i,l,r}dS=\int_{S_{l+1}}A_{i,l,r}dS=0.
	\end{align*}
	Choosing $\epsilon$ small enough such that $\{\nabla_i\}$ satisfy the condition of lemma \ref{Lpcur}, we may assume $\int_{\mathfrak{U}_l}|A_{i,l}|^2dV\le 2^{-2l}C\int_{\mathfrak{U}_l}|R^{\triangledown_{i,l}}|^2dV$. Assume 
	\begin{align*}
		\bigcup_{l=l_1+1}^{l_2-1}\mathfrak{U}_l\subset B_\delta\backslash B_{Rr_i^1}\subset\bigcup_{l=l_1}^{l_2}\mathfrak{U}_l.
	\end{align*}
	By Stokes' formula, we have
	\begin{align*}
		\int_{\mathfrak{U}_l}|R^{\triangledown_{i,l}}|^2dV=\int_{\mathfrak{U}_l}<\delta^{\triangledown_{i,l}}R^{\triangledown_{i,l}},A_{i,l}>+A_{i,l}\#A_{i,l}\#R^{\triangledown_{i,l}}dV+\int_{S_l}A_{i,l}\wedge\ast R^{\triangledown_{i,l}}-\int_{S_{l+1}}A_{i,l}\wedge\ast R^{\triangledown_{i,l}}.
	\end{align*}
    Let $\hat A_{i,l}(x)=2^{-l}A_{i,l}(2^{-l}x)$ and assume $\|R^{\hat\triangledown_{i,l}}\|_{L^\infty(\mathfrak{U}_0)}=2^{-2l}\|R^{\triangledown_{i,l}}\|_{L^\infty(\mathfrak{U}_l)}=2^{-2l}\|R^{\triangledown_i}\|_{L^\infty(\mathfrak{U}_l)}\le\gamma$ by choosing $\epsilon$ small enough, where $\gamma$ is the constant in lemma \ref{conn-cur}. Then we have
    \begin{align*}
    	\int_{\mathfrak{U}_l}A_{i,l}\#A_{i,l}\#R^{\triangledown_{i,l}}dV&\le C\|R^{\triangledown_{i,l}}\|_{L^\infty(\mathfrak{U}_l)}\int_{\mathfrak{U}_l}|A_{i,l}|^2dV\\
    	&\le C\|R^{\triangledown_{i,l}}\|_{L^\infty(\mathfrak{U}_l)}2^{-2l}\int_{\mathfrak{U}_0}|\hat A_{i,l}|^2dV\\
    	&\le C2^{2l}\epsilon\cdot2^{-2l}\int_{\mathfrak{U}_0}|R^{\hat\triangledown_{i,l}}|^2dV\\
    	&\le C\epsilon\int_{\mathfrak{U}_l}|R^{\triangledown_{i,l}}|^2dV.
    \end{align*}
    Choose $\epsilon$ small enough such that $C\epsilon\le\frac12$, then
	\begin{align*}
		&\sum_{l=l_1}^{l_2}\int_{\mathfrak{U}_l}|R^{\triangledown_{i,l}}|^2dV\\=&\sum_{l=l_1}^{l_2}\int_{\mathfrak{U}_l}\langle \left( \delta^{\triangledown_{i,l}} R^{\triangledown_{i,l}},A_{i,l} \rangle+A_{i,l}\#A_{i,l}\#R^{\triangledown_{i,l}}  \right)  dV+\int_{S_l}A_{i,l}\wedge\ast R^{\triangledown_{i,l}}-\int_{S_{l+1}}A_{i,l}\wedge\ast R^{\triangledown_{i,l}}\\
		\le&\int_{S_{l_1}}A_{i,l_1}\wedge\ast R^{\triangledown_{i,l_1}}-\int_{S_{l_2+1}}A_{i,l_2}\wedge\ast R^{\triangledown_{i,l_2}}+\sum_{l=l_1}^{l_2}\int_{\mathfrak{U}_l}\langle \delta^{\triangledown_{i,l}} R^{\triangledown_{i,l}},A_{i,l} \rangle  dV+\frac12\sum_{l=l_1}^{l_2}\int_{\mathfrak{U}_l}|R^{\triangledown_{i,l}}|^2dV.
	\end{align*}
    Then we have
    \begin{align*}
    	\int_{B_\delta\backslash B_{Rr_i^1}}|R^{\triangledown_i}|^2dV&\le\sum_{l=l_1}^{l_2}\int_{\mathfrak{U}_l}|R^{\triangledown_{i,l}}|^2dV\\
    	&\le2(\int_{S_{l_1}}A_{i,l_1}\wedge\ast R^{\triangledown_{i,l_1}}-\int_{S_{l_2+1}}A_{i,l_2}\wedge\ast R^{\triangledown_{i,l_2}}+\sum_{l=l_1}^{l_2}\int_{\mathfrak{U}_l}\langle \delta^{\triangledown_{i,l}} R^{\triangledown_{i,l}},A_{i,l} \rangle  dV).
    \end{align*}
	By equation (\ref{YMH}) and H{\"o}lder's inequality, we have
	\begin{align*}
		\sum_{l=l_1}^{l_2}\int_{\mathfrak{U}_l}\langle \delta^{\triangledown_{i,l}}R^{\triangledown_{i,l}},A_{i,l} \rangle  dV&\le\sum_{l=l_1}^{l_2}\|d^{\triangledown_{i,l}}u_{i,l}\|_{L^2(\mathfrak{U}_l)}\|u_{i,l}\|_{L^\infty(\mathfrak{U}_l)}\|A_{i,l}\|_{L^2(\mathfrak{U}_l)}\\
		&\le C\sum_{l=l_1}^{l_2}2^{-l}\|d^{\triangledown_{i,l}}u_{i,l}\|_{L^2(\mathfrak{U}_l)}\|R^\triangledown_{i,l}\|_{L^2(\mathfrak{U}_l)}\\
		&\le 2^{-l_1}C\epsilon,
	\end{align*}
	which tends to 0 as $\delta$ tends to 0 since $2^{-l_1}\le2\delta$. By Fatou's lemma, we have
	\begin{align*}
		\int_0^{R_0}\liminf_{i\to\infty}\int_{|x|=r}|R^{\triangledown_i}|^2dSdr\le K,
	\end{align*}
	hence
	\begin{align*}
		\delta\liminf_{\delta\to0}\int_{|x|=\delta}|R^{\triangledown_i}|^2dS=0
	\end{align*}
	Since $\|A_{i,l}\|_{L^\infty(S_l)}\le 2^{-l}C\|R^{\triangledown_i}\|_{L^\infty(S_l)}\le2^lC\epsilon$, we have
	\begin{align*}
		\int_{S_{l_1}}A_{i,l_1}\wedge\ast R^{\triangledown_{i,l_1}}\le\|A_{i,l_1}\|_{L^2(S_{l_1})}\|R^{\triangledown_{i,l_1}}\|_{L^2(S_{l_1})}\to0
	\end{align*}
	as $\delta\to0$. According to the same argument for $S_{l_2}$, we conclude that
	\begin{align*}
		\lim_{R\to\infty}\lim_{\delta\to0}\liminf_{i\to\infty}\int_{B_\delta\backslash B_{Rr_i^1}}|R^{\triangledown_i}|^2dV=0.
	\end{align*}
	Since lemma \ref{L^infty} and lemma \ref{concentration} shows that Higgs part $|d^{\triangledown_i}u_i|^2+\frac\lambda4(1-|u_i|^2)^2$ has no concentration point, we have
	\begin{align*}
		\lim_{R\to\infty}\lim_{\delta\to0}\liminf_{i\to\infty}\mathscr{A}(\nabla_i,u_i,B_\delta\backslash B_{Rr_i^1})=0.
	\end{align*}
\end{proof}
For any $\delta,R>0$, define $E_i=\{r_i\in(\frac{Rr_i^1}2,2\delta)\mid \mathscr{A}(\nabla_i,u_i;B_{2r_i}(x_0)\backslash B_{r_i}(x_0))\ge\epsilon\}$. If $\cup_{j}\cap_{i>j}E_i\ne\varnothing$, define equivalent classes of $\{(r_1,r_2,...)\mid r_i\in E_i\}$ such that $\{r_i\}$ equals to $\{r_i'\}$ if and only if $\liminf\frac{r_i}{r_i'}>0$ and $\limsup\frac{r_i}{r_i'}<\infty$. And define $\{r_i\}>\{r_i'\}$ if and only if $\liminf\frac{r_i'}{r_i}=0$. Similarly to the estimate of the number of element of $\Sigma$, we have the estimate of the number of equivalent class $L\le\frac{K}{\epsilon}$. Assume the equivalent classes are $\{r_i^1\}<\{r_i^2\}<...<\{r_i^L\}$. Note that
\begin{align*}
	&\mathscr{A}(\nabla_i,u_i;B_\delta(x_0)\backslash B_{Rr_i^1}(x_0))\\
	=&\sum_{j=2}^{L}\mathscr{A}(\nabla_i,u_i;B_{\delta r_i^j}(x_0)\backslash B_{Rr_i^{j-1}}(x_0))+\mathscr{A}(\nabla_i,u_i;B_{Rr_i^j}(x_0)\backslash B_{\delta r_i^j}(x_0))\\
	&+\mathscr{A}(\nabla_i,u_i;B_\delta(x_0)\backslash B_{Rr_i^L}(x_0)).
\end{align*}
For $j=1,...,l$, define
\begin{align*}
	\tilde A_{j,i}(x)&=r_i^jA_i(x_0+r_i^jx),\\
	\tilde u_{j,i}(x)&=u(x_0+r_i^jx).
\end{align*}
Since $\mathscr{A}(\nabla_i,u_i;B_{\delta r_i^j}(x_0)\backslash B_{Rr_i^{j-1}}(x_0))=\mathscr{A}(\widetilde\nabla_{j,i},\tilde u_{j,i};B_\delta(0)\backslash B_{R{r_i^{j-1}}/{r_i^j}}(0))$, where $\lim\limits_{i\to\infty}\frac{r_i^{j-1}}{r_i^j}=0$ and for any $r\in(\frac{Rr_i^{j-1}}{2r_i^j},2\delta)$, we have
\begin{align*}
	\liminf_{i\to\infty}\mathscr{A}(\widetilde\nabla_{j,i},\tilde u_{j,i};B_{2r}(0)\backslash B_r(0))<\epsilon
\end{align*}
(otherwise there exists an equivalent class between $\{r_i^{j-1}\}$ and $\{r_i^j\}$). By lemma \ref{neck}, we have
\begin{align*}
	\lim_{R\to\infty}\lim_{\delta\to0}\liminf_{i\to\infty}\mathscr{A}(\nabla_i,u_i;B_{\delta r_i^j}(x_0)\backslash B_{Rr_i^{j-1}}(x_0))=0.
\end{align*}
Similarly, $\{r_i^L\}$ is the largest equivalent implies $\{(\nabla_i,u_i)\}$ satisfy (\ref{ring_small}) by replacing $\{r_i^1\}$ with $\{r_i^L\}$ and thus 
\begin{align*}
	\lim_{R\to\infty}\lim_{\delta\to0}\liminf_{i\to\infty}\mathscr{A}(\nabla_i,u_i;B_\delta(x_0)\backslash B_{Rr_i^L}(x_0))=0.
\end{align*}
The uniformly $W^{1,2}$ bound of $\widetilde\nabla_{j,i}$ implies that there is a weakly limit $\widetilde\nabla_{j,\infty}$ in $\mathbb{R}^4-\{0\}$ and the removable singularity theorem shows that it could be extended to a Yang-Mills connection over $S^4$. The convergence may be not smooth in $S^4$, and we can repeat the bubble-neck decomposition at each blow-up point as above. This process must stop after finite steps by the uniform energy bound. Hence 
\begin{align*}
	\lim_{R\to\infty}\lim_{\delta\to0}\liminf_{i\to\infty}\mathscr{A}(\nabla_i,u_i;B_{Rr_i^j}(x_0)\backslash B_{\delta r_i^j}(x_0))=YM(\widetilde\nabla_{j,\infty})+\sum_{k=1}^{K_j}YM(\widetilde\nabla_{j,\infty,k}),
\end{align*}
where $\widetilde\nabla_{j,\infty,k}$ are the bubbles of $\widetilde\nabla_{j,\infty}$. For simplicity, we assume $\Sigma=\{x_0\}$. Finally, by choosing a subsequence, we have
\begin{align*}
	\lim_{i\to\infty}\mathscr{A}(\nabla_i,u_i)=&\lim_{\delta\to0}\lim_{i\to\infty}\mathscr{A}(\nabla_i,u_i;M\backslash B_\delta(x_0))+\lim_{R\to\infty}\lim_{i\to\infty}\mathscr{A}(\nabla_i,u_i;B_{Rr_i^1}(x_0))\\
	&+\lim_{R\to\infty}\lim_{\delta\to0}\lim_{i\to\infty}\mathscr{A}(\nabla_i,u_i;B_\delta(x_0)\backslash B_{Rr_i^1}(x_0))\\
	=&\mathscr{A}(\nabla_\infty,u_\infty)+YM(\widetilde\nabla_{1,\infty})+\sum_{j=1}^L(YM(\widetilde\nabla_{j,\infty})+\sum_{k=1}^{K_j}YM(\widetilde\nabla_{j,\infty,k})).
\end{align*}
Then we finish the prove.
\begin{remark}
	If we consider the Higgs fields taking values in $\Omega^0(\mathfrak{g}_E)$, we can get the similar energy identity. 
\end{remark}
 \begin{thm}
	Assume $\{(\nabla_i,\Phi_i)\}$ is a family of Yang-Mills-Higgs pairs  and $\mathscr{A}(\nabla_i,\Phi_i)\le K$, where $\Phi\in\Omega^0(\mathfrak{g}_E)$. Then there is a finite subset $\Sigma=\{x_1,...,x_l\}\subset M$, a Yang-Mills-Higgs pair $(\nabla_\infty,\Phi_\infty)$ on $M\backslash\Sigma$ and Yang-Mills connections  $\{\widetilde\triangledown_{jk}\mid1\le j\le l,1\le k\le K_j\}$ over $S^4$, such that there is a subsequence of $\{(\nabla_i,\Phi_i)\}$ converges to $(\nabla_\infty,\Phi_\infty)$ in $C^\infty_{loc}(M\backslash\Sigma)$ under gauge transformations and
	\begin{equation}
		\lim_{i\to\infty}\mathscr{A}(\nabla_i,\Phi_i)=\mathscr{A}(\nabla_\infty,\Phi_\infty)+\sum_{j=1}^l\sum_{k=1}^{K_j}YM(\widetilde\nabla_{jk}).
	\end{equation}
\end{thm}
\section*{Conflict interests}

There is no conflict of interest.

\section*{Acknowledgment}
All authors would like to thank Prof. Jiayu Li for his encouragement and constant help. The first author is supported by National Key R$\&$D Program of China 2022YFA1005400 and NFSC No.12031017 and the second author is supported by NSFC No.12001532.

\end{document}